\documentclass[12pt, letterpaper, twoside]{article}
\usepackage[utf8]{inputenc}

\date{September 8, 2016}
\usepackage{graphicx}
\usepackage{amssymb}
\usepackage{amsthm}
\usepackage{mathtools}
\newtheorem{definition}{Definition}[section]
\newtheorem{theorem}[definition]{Theorem}
\newtheorem{lemma}[definition]{Lemma}
\newtheorem{remark}[definition]{Remark}
\newtheorem{example}[definition]{Example}
\newtheorem{corollary}[definition]{Corollary}

\newtheorem{proposition}[definition]{Proposition}
\begin{document}

\title{The Probability of Primeness for Specially Structured Polynomial Matrices over Finite Fields with Applications to Linear Systems and Convolutional Codes\footnote{The final publication is available at http://link.springer.com/article/10.1007/s00498-017-0191-z}
}


\author{Julia Lieb          
}


\maketitle

\begin{abstract}
We calculate the probability that random polynomial matrices over a finite field with certain structures are right prime or left prime, respectively. 
In particular, we give an asymptotic formula for the probability that finitely many nonsingular polynomial matrices are mutually left coprime.
These results are used to estimate the number of reachable and observable linear systems as well as the number of non-catastrophic convolutional codes. Moreover, we are able to achieve an asymptotic formula for the probability that a parallel connected linear system is reachable.
\end{abstract}

\section{Introduction}
\label{intro}

Since the research work of Rosenbrock (see \cite{rb}) polynomial matrices over finite fields have played an important role when investigating discrete time linear systems.
In \cite{fu2}, Fuhrmann introduced the so-called polynomial model; a generalization of it, the so-called fractional model was developed in \cite{lo}.
As the transfer function of a linear system is a matrix of proper rational functions, it admits a coprime polynomial matrix fraction representation. This factorization was used by Fuhrmann and Helmke \cite{Fu-He15} to study networks of linear systems, especially their reachability and observability. In this connection, they proved criteria for these properties, which consist of coprimeness conditions on the polynomial matrices describing the node systems. In particular, they showed that a parallel connection of reachable systems is reachable if and only if the denominator matrices in the representation for the transfer functions of the node systems are mutually left coprime \cite{fuhr}, \cite{Fu-He15}.\\
Additionally, polynomial matrices are essential in the theory of convolutional codes. One could define a convolutional code as the image of a polynomial matrix, which right primeness is equivalent to the important property of the code to be non-catastrophic. On the other hand, it is possible to construct a convolutional code out of a linear system \cite{ros}, \cite{RY1999}. Rosenthal and York showed that such a code is represented by the corresponding linear system in a minimal way if and only if the system is reachable and that in this case, the code is non-catastrophic if and only if the system is observable, too \cite{RY1999}.\\
In \cite{hjl} and \cite{hjL}, a formula for the number of reachable linear systems over a finite field was proven via computing the number of polynomial matrices in Hermite form. In this article, we continue that work by calculating the number of minimal systems as well as the number of non-catastrophic convolutional codes based on an estimation for the number of right coprime polynomial matrix pairs. Moreover, we give an asymptotic expression for the number of mutually left coprime polynomial matrices, which enables us to calculate the probability that a parallel connected linear system is reachable.\\ 
The paper is structured as follows. We start with some definitions and preliminary results concerning linear systems, polynomial matrices and general counting strategies. In Section 3, we calculate the probability that a linear system is minimal. In Section 4, we prove the main theorem of this paper, Theorem \ref{mut}, which provides a formula for the probability of mutual left coprimeness. This enables us to compute the probability of reachability for a parallel connected linear system in Subsection 5.1. Finally, in Subsection 5.2., we calculate the probability of non-catastrophicity for a convolutional code.

\section{Preliminaries}
\label{pre}
\subsection{Linear Systems and Polynomial Matrices}
\label{lspm}

We start this subsection with some definitions and properties concerning polynomial matrices over an arbitrary field $\mathbb F$.

\begin{definition}\ \\
A polynomial matrix $Q\in\mathbb F[z]^{m\times m}$ is called \textbf{nonsingular} if $\det(Q(z))\not\equiv 0$. It is called \textbf{unimodular} if $\det(Q(z))\neq 0$ for all $z\in\overline{\mathbb F}$, i.e. if $\det(Q(z))$ is a nonzero constant. This is true if and only if $Q$ is invertible in $\mathbb F[z]^{m\times m}$. Thus, one denotes the group of unimodular $m\times m$-matrices over $\mathbb F[z]$ by $Gl_m(\mathbb F[z])$. Throughout this paper, $\mathbb F[z]$ should denote the ring of polynomial matrices and $\mathbb F(z)$ the field of rational functions with coefficients in $\mathbb F$.
\end{definition}

\begin{definition}\ \\
A polynomial matrix $H\in\mathbb F[z]^{p\times m}$ is called a \textbf{common left divisor} of $H_i\in\mathbb F[z]^{p\times m_i}$ for $i=1,\hdots,N$ if there exist matrices $X_i\in\mathbb F[z]^{m\times m_i}$ with $H_i(z)=H(z)X_i(z)$ for $i=1,\hdots,N$. It is called a \textbf{greatest common left divisor}, which is denoted by $H=\operatorname{gcld}(H_1,\hdots,H_N)$, if for any other common left divisor $\tilde{H}\in\mathbb F[z]^{p\times\tilde{m}}$, there exists $S(z)\in\mathbb F[z]^{\tilde{m}\times m}$ with $H(z)=\tilde{H}(z)S(z)$.\\
A polynomial matrix $E\in\mathbb F[z]^{p\times m}$ is called a \textbf{common left multiple} of $E_i\in\mathbb F[z]^{m_i\times m}$ for $i=1,\hdots,N$ if there exist matrices $X_i\in\mathbb F[z]^{p\times m_i}$ with $X_i(z)E_i(z)=E(z)$ for $i=1,\hdots,N$. It is called a \textbf{least common left multiple}, which is denoted by $E=\operatorname{lclm}(E_1,\hdots,E_N)$, if for any other common left multiple $\tilde{E}\in\mathbb F[z]^{\tilde{p}\times m}$, there exists $R(z)\in\mathbb F[z]^{\tilde{p}\times p}$ with $R(z)E(z)=\tilde{E}(z)$.\\
One defines a \textbf{(greatest) common right divisor}, which is denoted by $\operatorname{gcrd}$, and a \textbf{(least) common right multiple}, which is denoted by $\operatorname{lcrm}$, analoguely.
\end{definition}

\begin{definition}\ \\
Polynomial matrices $H_i\in\mathbb F[z]^{p\times m_i}$ are called \textbf{left coprime} if there exists $X\in\mathbb F[z]^{m\times p}$ such that $H=\operatorname{gcld}(H_1,\hdots,H_N)$ satisfies $HX=I_p$. In particular, one polynomial matrix $H\in\mathbb F[z]^{p\times m}$ is called \textbf{left prime} if there exists $X\in\mathbb F[z]^{m\times p}$ with $HX=I_p$. Analoguely, one defines the property to be \textbf{right coprime} or \textbf{right prime}, respectively. Note that in the case $p=m$, right primeness and left primeness are equivalent to the property to be unimodular.
\end{definition}

For our probability estimations later in this work, we will mainly use the following characterization of coprimeness.

\begin{theorem}\cite[Theorem 2.27]{Fu-He15}\ \\
\vspace{-6mm}
\begin{itemize}
\item[(a)]
The polynomial matrices $H_i\in\mathbb F[z]^{p\times m_i}$ are left coprime if and only if $\operatorname{rk}([H_1(z)\cdots H_N(z)])=p$ for all $z\in\overline{\mathbb F}[z]$.
\item[(b)]
The polynomial matrices $H_i\in\mathbb F[z]^{p_i\times m}$ are right coprime if and only if $\operatorname{rk}\begin{pmatrix}H_1(z)\\ \vdots\\ H_N(z) \end{pmatrix}=m$ for all $z\in\overline{\mathbb F}[z]$.
\end{itemize}
\end{theorem}

For parallel connections of linear systems, the following property will be crucial.

\begin{definition}\ \\
Nonsingular polynomial matrices $D_1,\hdots, D_N\in\mathbb F[z]^{m\times m}$ are called \textbf{mutually left coprime} if for each $i=1,\hdots,N$, $D_i$ is left coprime with $\operatorname{lcrm}\{D_j\}_{j\neq i}$.
\end{definition}

This criterion for mutually left coprimeness is not very easy to handle. Thus, we will employ an equivalent characterization to prove Theorem \ref{mut}, the main result of Section \ref{mlc}.

\begin{theorem}\cite[Proposition 10.3]{Fu-He15}\label{mutcrit} \\
Nonsingular polynomial matrices $D_1,\hdots, D_N\in\mathbb F[z]^{m\times m}$ are mutually left coprime if and only if
$$\mathcal{D}_N:=\left[\begin{array}{cccc}
D_1 & D_2 &  & 0 \\ 
 & \ddots & \ddots &  \\ 
0 &  & D_{N-1} & D_N
\end{array}\right]$$
 is left prime.
\end{theorem}

Finally, in Section \ref{rp}, we will need the following well-known criterion for coprimeness of scalar polynomials.

\begin{theorem}\label{syl}\ \\
Two polynomials $p(z)=\sum_{i=0}^mp_iz^{i}$ and $q(z)=\sum_{i=0}^nq_iz^{i}$  are coprime if and only if the Sylvester resultant 
$$\operatorname{Res}(p,q):=\left[\begin{array}{cccccc}
p_0 &  &  & q_0 &  &  \\ 
\vdots & \ddots &   & \vdots & \ddots &   \\ 
p_{m} &   & p_0 & q_n &   & q_0 \\ 
  & \ddots & \vdots &   & \ddots & \vdots \\ 
  &   & p_m &   &   & q_n
\end{array}\right]\in\mathbb F^{(n+m)\times(n+m)}$$
is invertible.
\end{theorem}

For the second part of this subsection, we consider discrete-time linear control systems of the form
\begin{align}\label{eq:linsys}
x(\tau +1)&=Ax(\tau)+Bu(\tau) \nonumber\\
y(\tau)&=Cx(\tau)+Du(\tau)
\end{align}
with $A\in \mathbb{F}^{n\times n}, B\in
 \mathbb{F}^{n\times m}, C\in\mathbb F^{p\times n}, D\in\mathbb F^{p\times m}$,  input $u\in\mathbb F^m$, state vector $x\in\mathbb F^n$, output $y\in\mathbb F^p$ and $\tau\in\mathbb N_0$.\\
In the following, we will frequently identify this system with the matrix-quadruple $(A,B,C,D)$. Moreover, we denote by $T(z)=C(zI-A)^{-1}B+D$ the \textbf{transfer function} of the system and by $\delta(T)$ its \textbf{McMillan degree}.

\begin{theorem}\cite[Theorem 2.29]{Fu-He15}\ \\
Let $T\in\mathbb F(z)^{p\times m}$ be arbitrary. Then, there exist right coprime polynomial matrices $P\in\mathbb F[z]^{p\times m}$ and $Q\in\mathbb F[z]^{m\times m}$ nonsingular such that $T(z)=P(z)Q(z)^{-1}$.\\
If $\tilde{P}\in\mathbb F[z]^{p\times m}$ and $\tilde{Q}\in\mathbb F[z]^{m\times m}$ are right coprime with $\tilde{Q}$ nonsingular such that $\tilde{P}(z)\tilde{Q}(z)^{-1}=T(z)=P(z)Q(z)^{-1}$, then there exists a (unique) unimodular matrix $U\in Gl_m(\mathbb F[z])$ with $\tilde{P}=PU$ and $\tilde{Q}=QU$.
\end{theorem}

Amongst this set of unimodular equivalent right coprime factorizations, we focus on two particular choices, where the denominator matrix has some special properties. To this end, we first need the following definitions.

\begin{definition}\ \\
The \textbf{$j$-th column degree} of a polynomial matrix $H(z)\in\mathbb F[z]^{p\times m}$ is defined as $\nu_j:=\deg_jH:=\max_{1\leq i\leq p}\deg(h_{ij})$. Furthermore, let $[h_{ij}]$ denote the coefficient of $z^{\nu_j}$ in $h_{ij}$. Then, the \textbf{highest column degree coefficient matrix} $[H]_{hc}\in\mathbb F^{p\times m}$ is defined as the matrix consisting of the entries $[h_{ij}]$. For $p=m$, one calls $H$ \textbf{column proper} if $[H]_{hc}\in Gl_m(\mathbb F)$.
\end{definition}

\begin{definition}\cite[Corollary 2.42]{Fu-He15}, \cite[Proposition 5.1]{hin}\ \\
Let $Q\in\mathbb F[z]^{m\times m}$ be nonsingular. Then, there exist
\vspace{-2mm}
\begin{itemize}
\item[(a)]
a unimodular matrix $U_1\in Gl_m(\mathbb F[z])$ such that
$$QU_1=Q^H:=\left[\begin{array}{cccc}
q_{11}^{(H)} & 0 & \hdots & 0 \\ 
\vdots & \ddots & \ddots & \vdots \\ 
\vdots &   & \ddots & 0 \\ 
q_{m1}^{(H)} & \hdots & \hdots & q_{mm}^{(H)}
\end{array}\right]$$
with $q_{ii}^{(H)}$ monic and $\deg{q_{ij}^{(H)}}<\deg{q_{ii}^{(H)}}=:\kappa_{m+1-i}$ for $1\leq j<i\leq m$.\\
$Q^H$ is unique and is called \textbf{Hermite canonical form}. Moreover, $Q^H$ is called of \textbf{simple form} if $\kappa_j=0$ for $j\geq 2$.
\item[(b)]
a unimodular matrix $U_2\in Gl_m(\mathbb F[z])$ such that
$$QU_2=Q^{KH}:=\left[\begin{array}{ccc}
q_{11}^{(KH)} &  \hdots & q_{1m}^{(KH)} \\ 
\vdots &  & \vdots \\ 
q_{m1}^{(KH)} &  \hdots & q_{mm}^{(KH)}
\end{array}\right]$$
with $q_{ii}^{(KH)}$ monic, $\deg{q_{ij}^{(KH)}}<\deg{q_{ii}^{(KH)}}$ for $j\neq i$, $\deg{q_{ji}^{(KH)}}<\deg{q_{ii}^{(KH)}}$ for $j<i$ and  $\deg{q_{ji}^{(KH)}}\leq\deg{q_{ii}^{(KH)}}$ for $j>i$.\\
$Q^{KH}$ is unique and is called \textbf{Kronecker-Hermite canonical form}. Note that it is always column proper.
\end{itemize}
\end{definition}

\begin{theorem}\label{form}\ \\
Let $T(z)=P(z)Q(z)^{-1}$ with $P\in\mathbb F[z]^{p\times m}, Q\in\mathbb F[z]^{m\times m}, \det(Q)\not\equiv 0$ a right coprime factorization of the transfer function. Then, it holds:
\begin{itemize}
\item[(a)]
$\delta(T)=\deg(\det(Q))$ \cite[Theorem 4.24]{Fu-He15}.
\item[(b)]
For every unimodular matrix $U\in Gl_m(\mathbb F[z])$, the pair $(PU,QU)$ is also a right coprime factorization of the corresponding transfer function. Consequently, one could either assume that $Q=Q^{KH}$ or that $Q=Q^{H}$. 
\end{itemize}
\end{theorem}


So far, we only focused on the structure of the denominator matrix $Q$. But if it is in Kronecker-Hermite from, i.e. in particular, column proper, one also has some knowledge about the nominator matrix $P$. 

\begin{lemma}\label{degn}\cite[Proposition 2.30]{Fu-He15}\ \\
Let $(A,B,C,D)\in\mathbb F^{n\times n}\times\mathbb F^{n\times m}\times\mathbb F^{p\times n}\times\mathbb F^{p\times m}$ and $C(zI-A)^{-1}B+D=P(z)Q(z)^{-1}$ with $P\in\mathbb F[z]^{p\times m}, Q\in\mathbb F[z]^{m\times m}, \det(Q)\not\equiv 0$ and $Q$ column proper. Then, one has for $j=1,....,m$:
$$\deg_jP(z)\leq\deg_jQ(z)\qquad \text{and}\qquad \deg_jP(z)<\deg_jQ(z)\ \ \text{if}\ \ D=0.$$
\end{lemma}

The aim of this article is to achieve probability results by counting the number of coprime polynomial matrix factorizations with special properties. Therefore, in the following subsection, we list some basic counting formulas, which will be useful for these purposes.

\subsection{General Counting Strategies}
\label{cs}

To compute the probability that a mathematical object has a special property, it is necessary to count mathematical objects. Therefore, in the following, we restrict our considerations to a finite field $\mathbb{F}$, which is endowed with the uniform probability distribution that assigns to each field element the same probability
$$t=\frac{1}{|\mathbb{F}|}$$ 
and denote the corresponding probability of a set $A$ by $\operatorname{Pr}(A)$.\\

For our computations, we will need the following lemmata. Those which are not proven here are well-known formulas.

\begin{lemma}\label{glcar}\ \\
The number of invertible $n\times n$-matrices over $\mathbb F$ is equal to
$$|Gl_n(\mathbb F)|=t^{-n^2}\prod_{j=1}^n(1-t^j).$$
\end{lemma}

\begin{lemma}\label{ine}(Inclusion-Exclusion Principle)\ \\
Let $A_1,\hdots,A_n$ be finite sets and $X=\bigcup_{i=1}^n A_i$. For $I\subset\{1,\hdots,n\}$, define $A_I:=\bigcap_{i\in I}A_i$. Then, it holds
$$|X|=\sum_{\emptyset\neq I\subset\{1,\hdots,n\}}(-1)^{|I|-1}|A_I|.$$
\end{lemma}

\begin{lemma}\label{2coprime}\cite{gar} \\
The probability that $N$ monic polynomials $d_1,\hdots,d_N\in\mathbb F[z]$ with $\deg(d_i)=n_i\in\mathbb N$ for $i=1,\hdots, N$ are coprime is equal to $1-t^{N-1}$.
\end{lemma}

\begin{definition}\label{simple}\ \\
For $n_1,\hdots,n_N\in\mathbb N$, let $X(n_1,\hdots,n_N)$ be the set of all $N$-tuples of matrices $D_i\in\mathbb F[z]^{m\times m}$ in Hermite form with $\deg(D_i)=n_i$ for $i=1,\hdots,N$. Moreover, denote by $\kappa_m^{(i)},\ldots,\kappa_{1}^{(i)}$ the row degrees of $D_i$, i.e. the $(j,j)$-entry of $D_i$ has degree $\kappa_{m-j+1}^{(i)}$ and $\kappa_m^{(i)}+\cdots+\kappa_{1}^{(i)}=n_i$. Furthermore, for\\
$\kappa=(\kappa_m^{(1)},\hdots,\kappa_{1}^{(1)},\hdots,\kappa_m^{(N)},\hdots,\kappa_1^{(N)})$, let $X_{\kappa}(n_1,\hdots,n_N)$ be the subset of $X(n_1,\hdots,n_N)$ for which the row degrees are equal to $\kappa$. Finally, one calls $\mathcal{D}_N=\left[\begin{array}{cccc}
D_1 & D_2 & 0 & 0 \\ 
0 & \ddots & \ddots & 0 \\ 
0 & 0 & D_{N-1} & D_N
\end{array}\right]$ of simple form if $\kappa_j^{(i)}=0$ for $j\geq 2$ and $1\leq i\leq N$.
\end{definition}

\begin{lemma}\label{lcodim}\ \\
The cardinality of $X_{\kappa}(n_1,\hdots,n_N)$ is equal to 
$$\prod_{i=1}^N\prod_{j=1}^{m} t^{-(m-j+1)\cdot \kappa_{j}^{(i)}}=t^{-m(n_1+\cdots n_N)}\prod_{i=1}^N\prod_{j=1}^{m} t^{(j-1)\cdot \kappa_{j}^{(i)}}$$
and the cardinality of 
$X(n_1,\hdots,n_N)$ is equal to
\begin{align*}
t^{-m(n_1+\cdots+n_N)}\prod_{i=1}^N\sum_{\kappa_1^{(i)}+\cdots+\kappa_m^{(i)}
=n_i}\prod_{j=1}^{m} t^{(j-1)\cdot \kappa_{j}^{(i)}}
=t^{-m(n_1+\cdots+n_N)}(1+O(t)).
\end{align*}
Consequently, it holds
\begin{equation}\label{codim}
\frac{|X_{\kappa}(n_1,\hdots,n_N)|}{|X(n_1,\hdots,n_N)|}=t^{c_{\kappa}}\cdot  (1-O(t))\quad \text{with}\quad c_{\kappa}=\sum_{i=1}^N\sum_{j=1}^m(j-1)\kappa^{(i)}_j.
\end{equation}
In particular, one has $c_{\kappa}=0$ if $\mathcal{D}_N$ is of simple form.
\end{lemma}

\begin{proof}\ \\
The $j-1$ polynomials beyond the diagonal of $D_i$ in row $j$ are of degree less than $\kappa_{m-j+1}^{(i)}$, which means that one has $t^{-\kappa^{(i)}_{m-j+1}}$ possibilities for each of them. For the monic polynomial on the diagonal of row $j$, one has $t^{-\kappa^{(i)}_{m-j+1}}$ possibilities, too. Thus, the set $X_{\kappa}(n_1,\hdots,n_N)$ has cardinality 
\begin{align*}
\prod_{i=1}^N\prod_{j=1}^mt^{-j\cdot \kappa^{(i)}_{m-j+1}}&=\prod_{i=1}^N\prod_{j=1}^{m} t^{-(m-j+1)\cdot \kappa_{j}^{(i)}}=\prod_{i=1}^Nt^{-m\sum_{j=1}^m\kappa_{j}^{(i)}}\prod_{j=1}^{m} t^{(j-1)\cdot \kappa_{j}^{(i)}}=\\
&=t^{-m(n_1+\cdots n_N)}\prod_{i=1}^N\prod_{j=1}^{m} t^{(j-1)\cdot \kappa_{j}^{(i)}}
\end{align*}
because $n_i=\sum_{j=1}^m\kappa_j^{(i)}$. The formula for $|X(n_1,\hdots,n_N)|$ follows by summing over all possible values for $\kappa$. For the asymptotic result, one employs that $\prod_{j=1}^{m} t^{(j-1)\cdot \kappa_{j}^{(i)}}=1$ for simple form and $\prod_{j=1}^{m} t^{(j-1)\cdot \kappa_{j}^{(i)}}=O(t)$, otherwise.
%
%
\end{proof}

\begin{lemma}\label{anzirred}\ \\
The number of monic irreducible polynomials in $\mathbb F[z]$ of degree $j$ is equal to 
$$\varphi_j=\frac{1}{j}\sum_{d\mid j}\mu(d)t^{-j/d}=\frac{1}{j}t^{-j}+O(t^{-(j-1)})$$
where $\mu$ counts the number of distinct prime factors of an integer and is zero if the integer is the multiple of a square-number.
\end{lemma}

\begin{remark}\label{z0}\ \\
If one denotes by $f_{z_0}$ the minimal polynomial of $z_0\in\overline{\mathbb F}$ over $\mathbb F$ and sets $g_{z_0}:=\deg(f_{z_0})$, then for $g\in\mathbb N$, the number of $z_0\in\overline{\mathbb F}$ with $g_{z_0}=g$ is at most $\varphi_g\cdot g=O(t^g)$ since there are $\varphi_g$ possible minimal polynomials for $z_0$ and each of them has at most $g$ different zeros. In particular, for $g=1$, this number is equal to $t$.
\end{remark}


\begin{lemma}\label{res}\ \\
Let $z_0, z_1\in\overline{\mathbb F}$ with $z_0\neq z_1$ 
as well as 
$n\in\mathbb N$ be fixed.
Then, it holds:
\vspace{-2mm}
\begin{itemize}
\item[(a)] 
The number of $d\in\mathbb F[z]$ monic with $\deg(d)=n$ such that $d(z_0)=0$ is equal to $t^{-n+g_{z_0}}$ if $n\geq g_{z_0}$ and zero if $n<g_{z_0}$. Moreover, the number of $d\in\mathbb F[z]$ monic with $\deg(d)=n$ such that $d(z_0)=d(z_1)=0$ is equal to $t^{-n+\deg(\operatorname{lcm}(f_{z_0},f_{z_1}))}$ if $n\geq\deg(\operatorname{lcm}(f_{z_0},f_{z_1}))$ and zero otherwise. In particular, for $z_0, z_1\in\mathbb F$, it is equal to $t^{-n+2}$ if $n\geq 2$ and zero if $n=1$.
\item[(b)]
Let $w, \tilde{w}\in\mathbb F(z_0)[z]$ with $\tilde{w}(z_0)\neq 0$ be fixed.
Then, the number of $d\in\mathbb F[z]$ monic with $\deg(d)=n$ such that $w(z_0)=\tilde{w}(z_0)\cdot d(z_0)$ is at most $t^{-n+1}$. Moreover, the number of $d\in\mathbb F[z]$ with $\deg(d)<n$ such that $w(z_0)=\tilde{w}(z_0)\cdot d(z_0)$ is at most $t^{-n+1}$. In particular, for $z_0\in\mathbb F$, it is equal to $t^{-n+1}$ in both cases.
\item[(c)]
Let $w, \tilde{w}\in\mathbb F(z_0,z_1)[z]$ with $\tilde{w}(z_0)\neq 0\neq\tilde{w}(z_1)$ be fixed.
Then, for $n\geq 2$, the number of $d\in\mathbb F[z]$ with $\deg(d)<n$ such that $w(z_0)=\tilde{w}(z_0)\cdot d(z_0)$ and $w(z_1)=\tilde{w}(z_1)\cdot d(z_1)$ is at most $t^{-n+2}$. In particular, for $z_0, z_1\in\mathbb F$, it is equal to $t^{-n+2}$.
\end{itemize}
\end{lemma}

\begin{proof}
\begin{itemize}
\item[(a)]
It holds $d(z_0)=0$ if and only if $f_{z_0}$ divides $d(z)$. Thus, one has to count the number of degree $n$ monic multiples of $f_{z_0}$, which coincides with the number of monic polynomials in $\mathbb F[z]$ of degree $n-g_{z_0}$ if the last expression is non-negative; otherwise $f_{z_0}$ cannot divide $d$. Therefore, one has $t^{-(n-g_{z_0})}$ possibilities for $d$ if $n\geq g_{z_0}$ and if $n<g_{z_0}$, the number of possibilities is equal to zero.\\ 
For the second part of statement (a), one has the condition that $\operatorname{lcm}(f_{z_0},f_{z_1})$ has to divide $d$, which could be treated with a similar argumentation as above. Note that there are only the two possibilities $\operatorname{lcm}(f_{z_0},f_{z_1})=f_{z_0}=f_{z_1}$ and $\operatorname{lcm}(f_{z_0},f_{z_1})=f_{z_0}\cdot f_{z_1}$ because $f_{z_0}$ and $f_{z_1}$ are irreducible. Since $z_0\neq z_1$, for $z_0, z_1\in\mathbb F$, one has $\operatorname{lcm}(f_{z_0},f_{z_1})=(z-z_0)(z-z_1)$.
Thus, for $n=1$, the number of possibilities is equal to zero and for $n\geq 2$, there are $t^{-(n-2)}$ possibilities for $d$.
\item[(b)]
If $d$ is fixed to $w(z_0)/\tilde{w}(z_0)\in\overline{\mathbb F}$ at $z_0$, one could choose all coefficients of $d$ but the constant one randomly and then, solve the corresponding equation with respect to this constant coefficient. Therefore, it is fixed by the other coefficients, which leads to a factor of at most $t$ for the number of possibilities. Note that if $z_0\notin\mathbb F$, for some random choices, one gets a value for the constant coefficient that is not in $\mathbb F$ and thus, not all choices for the other coefficients are possible. But this only decreases the number of possibilities. Thus, one has at most $t^{-n+1}$ possibilities for $d$. If $z_0\in\mathbb F$, all choices for the other coefficients are possible and hence, one has exactly $t^{-n+1}$ possibilities.
\item[(c)]
Denote by $a_0,\hdots,a_{n-1}$ the coefficients of $d$. If one chooses $a_2,\hdots,a_{n-1}$ arbitrarily, one gets a system of two linear equations of the form
$$\left[\begin{array}{cc}
z_0 & 1 \\ 
z_1 & 1
\end{array}\right]\cdot \left(\begin{array}{c}
a_1 \\ 
a_0
\end{array}\right)= \left(\begin{array}{c}
y_1 \\ 
y_2
\end{array}\right) $$
where $y_1$ and $y_2$ depend on $w$, $\tilde{w}$, $z_0$, $z_1$ and $a_2,\hdots,a_n$. Since $\det\left[\begin{array}{cc}
z_0 & 1 \\ 
z_1 & 1
\end{array}\right]=z_0-z_1\neq 0$, there exists a unique solution for $a_0$ and $a_1$. Hence, these two coefficients are fixed by the others which gives a factor of $t^2$ for the number of possibilities. As in part $(b)$, it is not clear that one gets values for $a_0$ and $a_1$ that are elements of $\mathbb F$. Therefore, the number of possibilities is at most $t^{-n+2}$. For $z_0, z_1\in\mathbb F$, one gets $a_0, a_1\in\mathbb F$, and hence, one has exactly $t^{-n+2}$ possibilities.
\end{itemize}
\end{proof}
%
At the end of this section, a method should be introduced, which will be applied several times througout this article.

\begin{lemma}\label{it}(Method of Iterated Column/Row Operations)\ \\
Let $G\in\mathbb F[z]^{n\times m}$ and $z_0\in\overline{\mathbb F}$ with $\operatorname{rk}(G(z_0))<\min(n,m)$.
\begin{itemize}
\item[(a)]
If $m<n$, there exist $k\in\{0,\hdots,m-1\}$, a set of row indices $\{i_1,\hdots,i_k\}\subset\{1,\hdots,n\}$ and values $\lambda_{r}\in\mathbb F(z_0)$, which (only) depend on entries $g_{ij}$ of $G$ with $i\in\{i_1,\hdots,i_k\}$ and on $z_0$, such that
\begin{align}\label{lambda}
g_{i,m-k}(z_0)=\sum_{r=m-k+1}^mg_{ir}(z_0)\cdot\lambda_{r}\qquad\text{for}\ \  i\in\{1,\hdots,n\}\setminus\{i_1,\hdots,i_k\}.
\end{align}
 \item[(b)]
If $n<m$, there exist $k\in\{1,\hdots,n\}$, a set of column indices $\{j_1,\hdots,j_{k-1}\}\subset\{1,\hdots,m\}$ and values $\lambda_{r}\in\mathbb F(z_0)$, which (only) depend on entries $g_{ij}$ of $G$ with $j\in\{j_1,\hdots,j_{k-1}\}$ and on $z_0$, such that
\begin{align}\label{lambda2}
g_{kj}(z_0)=\sum_{r=1}^{k-1}g_{rj}(z_0)\cdot\lambda_{r}\qquad\text{for}\ \  j\in\{1,\hdots,m\}\setminus\{j_1,\hdots,j_{k-1}\}. 
\end{align}
\end{itemize}
\end{lemma}

\begin{proof}
\begin{itemize}
\item[(a)]
Set $p:=n-m$.
If $$\operatorname{rk}\left[\begin{array}{ccc}
g_{11} & \hdots & g_{1m} \\ 
\vdots   &  & \vdots \\ 
g_{m+p,1} & \hdots & g_{m+p,m}
\end{array}\right](z_0)<m, $$
than either $g_{1m}(z_0)=\cdots=g_{m+p,m}(z_0)=0$, which implies that equations \eqref{lambda} are fulfilled for $k=0$ and one is done, or it is possible to choose a nonzero entry from the set $\{g_{1m}(z_0),\hdots,g_{m+p,m}(z_0)\}$. In this case, one chooses the nonzero entry with the least row index, which should be denoted by $i_1$. Then, one subtracts the last column times $\frac{g_{i_1,j}(z_0)}{g_{i_1,m}(z_0)}$ from the $j$-th column for $j=1,\hdots, m-1$, which nullifies the $i_1$-th row but its last entry. Afterwards, this changed $i_1$-th row is used to nullify the other entries of the last column by adding appropriate multiplies of it to the other rows. Note that this final step only changes the last column of $G$.  
Define $G^{(1)}\in\mathbb F[z]^{(m+p)\times m}$ by 
$$g^{(1)}_{ij}=\begin{cases}
g_{ij} &  \text{for}\quad i= i_1,\ j= m\\
g_{ij}-g_{im}\cdot\frac{g_{i_1j}}{g_{i_1m}} &  \text{otherwise}
\end{cases}.$$
Then, it holds $m>\operatorname{rk}(G(z_0))=\operatorname{rk}(G^{(1)}(z_0))$.
One iterates this procedure, i.e. if column $m-1$ of $G^{(1)}(z_0)$ contains an entry that is unequal to zero, one uses it to nullify its row, whose index should be denoted by $i_2$, and afterwards its column. Setting $G^{(0)}:=G$, this leads to a sequence of matrices $G^{(k)}\in\mathbb F[z]^{(m+p)\times m}$ with $\operatorname{rk}(G^{(k)})(z_0)<m$ for $0\leq k\leq m-1$, which is obtained by the recursion formula
$$g^{(k)}_{ij}=\begin{cases} 
g_{ij}^{(k-1)} &  \text{for}\quad i= i_k,\ j= m-k+1\\
g_{ij}^{(k-1)}-g^{(k-1)}_{i,m-k+1}\cdot\frac{g^{(k-1)}_{i_k,j}}{g^{(k-1)}_{i_k,m-k+1}} &  \text{otherwise}
\end{cases}.$$
One stops this iteration when all entries of column $m-k$ of $G^{(k)}$ are zero at $z_0$. Note that the last $k$ columns of $G^{(k)}(z_0)$ are linearly independent since for $j=0,\hdots,k-1$, it holds $g^{(k)}_{i_{j+1},m-j}(z_0)\neq 0$ and $g^{(k)}_{i,m-j}\equiv 0$ for $i\neq i_{j+1}$, as well as $i_s\neq i_r$ for $r\neq s$, per construction. If the iteration does not stop in between, one ends up with the matrix $G^{(m-1)}$, whose entries $g^{(m-1)}_{i,1}$ for $i\in\{1,\hdots,n\}\setminus\{i_1,\hdots,i_{m-1}\}$ are not zero per construction but have to be equal to zero at $z_0$ because of $\operatorname{rk}(G^{(m-1)})(z_0)<m$. Generally, if one stops with $G^{(k)}$, one has the conditions $g^{(k)}_{i,m-k}(z_0)=0$ for $i\in\{1,\hdots,n\}\setminus\{i_1,\hdots,i_k\}$. Using the recursion formula, this leads to 
$$g_{i,m-k}^{(k-1)}(z_0)-g^{(k-1)}_{i,m-k+1}(z_0)\cdot\frac{g^{(k-1)}_{i_{k},m-k}(z_0)}{g^{(k-1)}_{i_{k},m-k+1}(z_0)}=0\ \ \text{for}\ i\in\{1,\hdots,n\}\setminus\{i_1,\hdots,i_k\}.$$
Based on this formula, we will show per (reversed) induction with respect to $s$ that for $0\leq s\leq k$, there exist $\lambda_{r}^{(s)}\in\mathbb F(z_0)$, which only depend on entries of $G$ with row index contained in the set $\{i_1,\hdots,i_{k}\}$ as well as on $z_0$, such that
\begin{align}\label{s}
g_{i,m-k}^{(s)}(z_0)&=\sum_{r=m-k+1}^{m} g_{ir}^{(s)}(z_0)\cdot \lambda_{r}^{(s)}\quad\text{for}\ i\in\{1,\hdots,n\}\setminus\{i_1,\hdots,i_k\}.
\end{align}
The base clause $s=k$ is trivial since one already knows $g_{i,m-k}^{(k)}(z_0)=0$ for $i\in\{1,\hdots,n\}\setminus\{i_1,\hdots,i_k\}$. Now, one assumes that the statement is valid for $s$ and considers the case $s-1$. One obtains, 
\begin{align*}
g_{i,m-k}^{(s-1)}(z_0)&=g_{i,m-k}^{(s)}(z_0)+g_{i,m-s+1}^{(s-1)}(z_0)\cdot\frac{g_{i_{s},m-k}^{(s-1)}(z_0)}{g_{i_s,m-s+1}^{(s-1)}(z_0)}=\\
&=\sum_{r=m-k+1}^{m} g_{ir}^{(s)}(z_0)\cdot \lambda_{r}^{(s)}+g_{i,m-s+1}^{(s-1)}(z_0)\cdot\frac{g_{i_{s},m-k}^{(s-1)}(z_0)}{g_{i_s,m-s+1}^{(s-1)}(z_0)}=\\
&=\sum_{r=m-k+1}^m \left(g_{ir}^{(s-1)}(z_0)-g_{i,m-s+1}^{(s-1)}(z_0)\cdot\frac{g_{i_{s},r}^{(s-1)}(z_0)}{g_{i_s,m-s+1}^{(s-1)}(z_0)}\right)\cdot \lambda_{r}^{(s)}+\\
&+g_{i,m-s+1}^{(s-1)}(z_0)\cdot\frac{g_{i_{s},m-k}^{(s-1)}(z_0)}{g_{i_s,m-s+1}^{(s-1)}(z_0)}=\\
&=\sum_{r=m-k+1}^{m} g_{ir}^{(s-1)}(z_0)\cdot \lambda_{r}^{(s-1)}
\end{align*}
with $\lambda_{r}^{(s-1)}:=\lambda_{r}^{(s)}$ for $r\neq m-s+1$ and 
$$\lambda_{m-s+1}^{(s-1)}:=\frac{g_{i_{s},m-k}^{(s-1)}(z_0)}{g_{i_s,m-s+1}^{(s-1)}(z_0)}-\sum_{m-s+1\neq r\geq m-k+1}\frac{g_{i_{s},r}^{(s-1)}(z_0)}{g_{i_s,m-s+1}^{(s-1)}(z_0)}\cdot \lambda_{r}^{(s)}.$$
Setting $s=0$ in \eqref{s}, completes the proof of part (a).
\item[(b)]
One could prove statement (b) analogously to statement (a). Instead of starting with the last column, one starts considering the first row of $G(z_0)$. If it is not identically zero, one chooses the nonzero entry with the largest column index. Then, one nullifies its row and column with a iteration procedure similar to part (a) but employing row operations instead of column operations.\\
Alternatively, one could apply part (a) to the matrix $G^T$ with inverse numbering of the rows.
\end{itemize}
\end{proof}

\section{Probability of Reachability and Observability}\label{rp}

The aim of this section is to calculate the probability that a linear system is reachable and observable, i.e. minimal.
In \cite{hjl}, Helmke et al. computed the probability that a linear system over a finite field is reachable and achieved the following formula:

\begin{theorem}\label{THMA}\cite[Theorem 1]{hjl}\ \\
The probability that a pair $(A,B)\in \mathbb{F}^{n\times n}\times \mathbb{F}^{n\times
  m}$ is reachable is equal to
\begin{equation}\label{eq:mainform1}
P_{n,m}(t)=\prod_{j=m}^{n+m-1} (1-t^{j})=1-t^m+O(t^{m+1}).
\end{equation}
\end{theorem} 
Using the duality between reachability and observability, one could easily deduce the probability of observability:
 
\begin{corollary}\label{obs}\ \\
The probability that a pair $(A,C)\in\mathbb F^{n\times n}\times\mathbb F^{p\times n}$ is observable is equal to 
$$\prod_{j=p}^{n+p-1} (1-t^{j})=1-t^p+O(t^{p+1}).$$
\end{corollary} 
 
 
The proof for Theorem 1 of \cite{hjl} uses that the number of reachable pairs is strongly connected with the number of matrices in Hermite form. Employing this correlation again, this theorem leads to the following corollary as well: 
 
\begin{corollary}\label{card}\ \\
The number of nonsingular polynomial matrices $Q\in\mathbb F[z]^{m\times m}$ in Hermite form whose determinant is a (monic) polynomial of degree $n$ is equal to 
\begin{align}\label{|}
H_{n,m}(\mathbb F)&=\sum_{\kappa_1+\cdots+\kappa_m=n}t^{-\sum_{i=1}^m(m-i+1)\cdot \kappa_{i}}=\frac{t^{-n^2-nm}\cdot P_{n,m}(t)}{|GL_n(\mathbb F)|}=\nonumber\\
&=t^{-mn}\prod_{j=1}^{n}\frac{1-t^{m+j-1}}{1-t^{j}}.
\end{align}
\end{corollary}

\begin{proof}\ \\
The first equality follows from Lemma \ref{lcodim}
and the second equation of \eqref{|} is part of the proof for Theorem 1 of \cite{hjl}. Finally, the third equation is a consequence of this theorem itself, i.e. of Theorem \ref{THMA} of this work, and of Theorem \ref{glcar}.
\end{proof}
 
\begin{remark}\label{anzkh}\ \\ 
Since both Hermite form and Kronecker-Hermite form are unique, the number of Kronecker-Hermite forms is equal to $H_{n,m}(\mathbb F)$, as well.
\end{remark}
 
In the remaining part of this section, we want to count the number of right coprime matrix pairs $(P,Q)$ as in Theorem \ref{form}, where we assume that $Q$ is in Kronecker-Hermite form to ensure that the factorization of the corresponding transfer function is unique. Since it seems very complicated to achieve an exact formula for the cardinality or probability, respectively, we investigate the asymptotic behaviour, when $1/t$ - the size of the field - tends to infinity.

\begin{definition}\label{M}\ \\
Let $M(p,n,m)$ be the set of all polynomial matrices $G=\left(\begin{array}{c} Q \\ P\end{array} \right)\in\mathbb F[z]^{(m+p)\times m}$ with $P\in\mathbb F[z]^{p\times m}$ and $Q\in\mathbb F[z]^{m\times m}$, where $Q$ is in Kronecker-Hermite form with $\deg(\det(Q))=n$ and $\deg_jP(z)\leq\deg_jQ(z)$ for $j=1,....,m$. Moreover, denote by $P^{rc}_{p,n,m}(t)$ the probability that $G\in M(p,n,m)$ is right prime.
\end{definition}

We continue with a lemma that enables us to write the probability of minimality as a product of the probability for right primeness with the probability of reachability.

\begin{lemma}\label{minl}\ \\
The probability that a linear discrete-time system described by $(A,B,C,D)\in\mathbb F^{n\times n}\times\mathbb F^{n\times m}\times\mathbb F^{p\times n}\times\mathbb F^{p\times m}$ is minimal is equal to
$P_{p,n,m}^{rc}(t)\cdot P_{n,m}(t)$.
\end{lemma}

\begin{proof}\ \\
For each $(A,B,C,D)\in\mathbb F^{n\times n}\times\mathbb F^{n\times m}\times\mathbb F^{p\times n}\times\mathbb F^{p\times m}$ describing a minimal system, there exists exactly one right coprime pair $(P,Q)\in\mathbb F[z]^{p\times m}\times\mathbb F[z]^{m\times m}$ where $Q$ is in Kronecker-Hermite form and $C(zI-A)^{-1}B+D=P(z)Q(z)^{-1}$. According to Theorem \ref{form} (a), $\deg(\det(Q))=n$ and according to Lemma \ref{degn}, it holds $\deg_jP(z)\leq\deg_jQ(z)$ for $j=1,....,m$. On the other hand, for every such pair $(P,Q)$, there exist exactly $|GL_n(\mathbb F)|$ minimal realizations $(A,B,C,D)$. Consequently, the number of minimal systems is equal to the number of pairs $(P,Q)$ times $|GL_n(\mathbb F)|$. According to Remark \ref{anzkh}, the number of Kronecker-Hermite forms is equal to $\frac{t^{-n^2-nm}\cdot P_{n,m}(t)}{|GL_n(\mathbb F)|}$. Moreover, for each of them, there are $\prod_{i=1}^m t^{-p(\kappa_i+1)}=t^{-p\sum_{i=1}^m(\kappa_i+1)}=t^{-p(n+m)}$ polynomial matrices $P\in\mathbb F[z]^{p\times m}$ which fulfill $\deg_jP(z)\leq\deg_jQ(z)$ for $j=1,....,m$. Consequently, the corresponding probability is equal to 
\begin{align*}
\frac{P_{p,n,m}^{rc}(t)\cdot t^{-(np+mp)}\cdot \frac{t^{-n^2-nm}\cdot P_{n,m}(t)}{|GL_n(\mathbb F)|}\cdot|GL_n(\mathbb F)|}{|\mathbb F^{n\times n}\times\mathbb F^{n\times m}\times\mathbb F^{p\times n}\times\mathbb F^{p\times m}|}=P_{p,n,m}^{rc}(t)\cdot P_{n,m}(t).
\end{align*}
\end{proof}

To achieve a formula for the probability of minimality, it remains to calculate $P^{rc}_{p,n,m}(t)$.

\begin{theorem}\label{1}\ \\
$$P^{rc}_{p,n,m}(t)=1-t^{p}+O(t^{p+1}).$$
\end{theorem}

\begin{proof}\ \\
We prove this theorem by computing the probability of the complementary set, i.e. the probability that there exists $z_0\in\overline{\mathbb F}$ such that $\operatorname{rk} (G(z_0))<m$. 
If $q_{ii}\equiv 1$ for some $i=1,\hdots,m$, all other elements in row $i$ are equal to zero. Hence, 
$\operatorname{rk}(G(z_0))=1+\operatorname{rk}(G_i(z_0))$,
where the index $i$ denotes the fact that the $i$-th row and column of $G$ are deleted. Consequently, one has to prove the statement for $m-1$ in this case. Thus, one could assume without restriction that all column degrees of $Q$ are unequal to zero, i.e. $Q$ has no constant diagonal elements. Hence, the matrix $G$ contains no fixed zeros, i.e. entries that have to be zero because of degree restrictions due to the Kronecker-Hermite form of $Q$. Moreover, no entries of $P$ are forced to be constant by those degree restrictions.\\
First, one considers $G^H:=\left(\begin{array}{c} Q^H \\ P^H\end{array} \right):=\left(\begin{array}{c} QU \\ PU\end{array} \right)=GU$, where $U$ is the unimodular matrix such that $Q^H$ is in Hermite form. Then, one applies the method of iterated column operations to $G^H$ (see Lemma \ref{it} (a)). Since $Q^H$ is lower triangular, the diagonal entries of it are not changed by this iteration process and the entries above the diagonal are identically zero, anyway. Thus, if the iteration stops after step $k$, one has $q_{m-k,m-k}^H(z_0)=0$ as first condition. 
The method of iterated column operations implies that 
if $q^H_{m-\tilde{k},m-\tilde{k}}(z_0)\neq 0$ for $\tilde{k}<k$, one chooses $i_{\tilde{k}+1}=m-\tilde{k}$. Therefore, if a row that belongs to $P^H$, i.e. with index greater than $m$, is nullified, one knows $q^H_{m-\tilde{k},m-\tilde{k}}(z_0)=0$. Define $I:=\{i_1,\hdots,i_{k}\}\cap\{i>m\}$. 
Then, one has the conditions $q_{m-j+1,m-j+1}^H(z_0)=0$ for $i_j\in I$ and $p^H_{i-m,m-k}(z_0)=\sum_{r=m-k+1}^{m} p^H_{i-m,r}(z_0)\cdot \lambda_{r}\quad \text{for}\ i\in\{m+1,\hdots,m+p\}\setminus I$ by Lemma \ref{it} (a).\\
For $P$, $Q$ and $Q^H$, one knows from Lemma \ref{res} that fixing several of the polynomial entries (that are not identically zero due to degree restrictions) at $z_0$ reduces the number of possible matrices at least by a factor of the form $t^h$ for some $h\in\mathbb N$. But unfortunately, one has no information about the effect of fixing polynomials of $P^H$ because one does not know anything about the possible degrees of the entries of this matrix. Thus, one has to switch back from $P^H$ to $P$.
Since $P^H=PU$, one has $p_{ij}^H=\sum_{l=1}^m p_{il}u_{lj}$. Inserting this into the above formula, leads to
$$\sum_{l=1}^mp_{i-m,l}(z_0)u_{l,m-k}(z_0)=\sum_{r=m-k+1}^{m}\sum_{l=1}^m p_{i-m,l}(z_0)u_{lr}(z_0)\cdot \lambda_{r},$$
which is equivalent to
\begin{align}\label{solve}
\sum_{l=1}^mp_{i-m,l}(z_0)\left(u_{l,m-k}(z_0)-\sum_{r=m-k+1}^{m}u_{lr}(z_0)\cdot \lambda_{r}\right)=0.
\end{align}
If $u_{l,m-k}(z_0)-\sum_{r=m-k+1}^{m}u_{lr}(z_0)\cdot \lambda_{r}=0$ for $l=1,\hdots,m$, columns $m-k,\hdots,m$ of $U$ were linearly dependent at $z_0$, which is a contradiction to the fact that $U$ is unimodular. Hence, there exists $l_0\in\{1,\hdots,m\}$ such that one could solve equation \eqref{solve} with respect to $p_{i-m,l_0}$.
Consequently, the $p-|I|$ polynomials $p_{i-m,l_0}$ with $i\in\{m+1,\hdots,m+p\}\setminus I$ are fixed at $z_0$ by $Q^H$ (which determines $Q$ and $U$) and the other entries of $P$. Note that $\lambda_{r}$ only depends on entries of $G^H$ whose row index is contained in the set $\{i_1,\hdots,i_{k}\}$ and hence, only on entries of $G$ whose row index belongs to $\{i_1,\hdots,i_{k}\}$ and on $U$.\\
Let $n_1,\hdots,n_{|I|}$ denote the degrees of the monic polynomials $q_{m-j+1,m-j+1}^H$ with $i_j\in I$ and $n_{|I|+1}$ the degree of $q^H_{m-k,m-k}$. Moreover, $n_{|I|+2}, \hdots, n_{p+1}$  should denote the maximal degrees of the $p-|I|$ fixed polynomial entries from $P$, which are not necessarily monic. Fix $g$ and $z_0$ with $g:=g_{z_0}\leq \min(n_1,\hdots,n_{|I|+1}):=n_{\min}$ as well as $Q^H$ such that $q^H_{m-k,m-k}(z_0)=0$ and $q_{m-j+1,m-j+1}^H(z_0)=0$ for $i_j\in I$. Then, $Q$ and $U$ are determined. Next, choose the polynomials $p_{{i-m},j}$ with $i\in I$ arbitrarily and define\\
$w_i:=-\sum_{l\neq l_0}p_{i-m,l}\left(u_{l,m-k}-\sum_{r=m-k+1}^{m}u_{lr}\cdot \lambda_{r}\right)$ for $i\in\{m+1,\hdots,m+p\}\setminus I$ as well as $\tilde{w}:=u_{l_0,m-k}-\sum_{r=m-k+1}^{m}u_{l_0,r}\cdot \lambda_{r}$.
Applying Lemma \ref{res} (a) and (b) as well as Remark \ref{z0}, one gets that the probability is at most 
$$\sum_{g=1}^{n_{\min}}\frac{g\cdot \varphi_g\cdot \prod_{i=1}^{|I|+1}t^{-n_i+g}\prod_{i=|I|+2}^{p+1}t^{-n_i-1+1}}{t^{-(n_1+n_2+\cdots+n_{p+1}+p-|I|)}}=O\left(\sum_{g=1}^{n_{\min}}t^p
\right)=O(t^p).$$
Furthermore, if one has the additional condition that $Q^H$ is not of simple form, the probability is even $O(t^{p+1})$. This is true since the probability that $G$ is not of simple form is $O(t)$ (see \eqref{codim}) and the considerations we made so far are valid for all values of $\kappa$, which is defined as in Definition \ref{simple}. Thus, it remains to consider the case that $Q^H$ is of simple form. Here, one has the condition:
$$\operatorname{rk}\left[\begin{array}{ccc}
I_{m-1} &  & 0 \\ 
q^H_{m1} & \hdots & q^H_{mm} \\ 
p^H_{11} & \hdots & p^H_{1m} \\ 
\vdots   &  & \vdots \\ 
p^H_{p1} & \hdots & p^H_{pm}
\end{array}\right](z_0)<m \Leftrightarrow q_{mm}^H(z_0)=p^H_{1m}(z_0)=\cdots=p^H_{pm}(z_0)=0.$$
Again, one has to switch back from $P^H$ to $P$. Doing this, one obtains the condition 
\begin{align}\label{u}
q_{mm}^H(z_0)=\sum_{l=1}^m p_{1l}u_{lm}(z_0)=\cdots=\sum_{l=1}^m p_{pl}u_{lm}(z_0)=0.
\end{align}
There are at most $g\cdot \varphi_g\cdot t^{-n+g}=O(t^{-n})$ possibilities for $z_0$ with $g_{z_0}=g$ and $q_{mm}^H$ monic with $\deg(q_{mm}^H)=n$ and $q_{mm}^H(z_0)=0$.
One fixes $z_0$ and $Q^H$ with these properties, which determines $U$ and $Q$ as well. Since $U$ is unimodular, there is a $l_0$ with $u_{l_0,m}(z_0)\neq 0$. Fix all entries of $P$ but those in column $l_0$ and set $u:=u_{l_0,m}$, $p^{(j)}:=p_{j,l_0}$ and $s^{(j)}:=\sum_{l\neq l_0}p_{jl}u_{lm}$ for $j=1,\hdots,p$. Moreover, denote by $f:=f_{z_0}$ the minimal polynomial of $z_0$. Then, one has the conditions $p^{(j)}\cdot u+s^{(j)}=f\cdot h^{(j)}$ for some $h^{(j)}\in\mathbb F[z]$ and $j=1,\hdots,p$. Note that here, $u$, $f$ and $s^{(j)}$ are already fixed. If one writes the involved polynomials as sums of monomials, one gets
$$\left(\sum_{i=0}^{\nu_{l_0}}p^{(j)}_iz^i\right)\cdot
\left(\sum_{i=0}^{\gamma}u_iz^i\right)+\left(\sum_{i=0}^{\beta_j}
s^{(j)}_iz^i\right)=\left(\sum_{i=0}^{g}f_iz^i\right)\cdot
\left(\sum_{i=0}^{\alpha_j}h^{(j)}_iz^i\right),$$
where the degrees $\gamma$ and $\beta_j$ are already fixed and  $\alpha_j=\max(\nu_{l_0}+\gamma,\beta_j)-g$. Equating coefficients, leads to
$$\underbrace{\left[\begin{array}{cccccc}
-u_0 &  &  & f_0 &  &  \\ 
\vdots & \ddots &   & \vdots & \ddots &   \\ 
-u_{\gamma} &   & -u_0 & \vdots &   & f_0 \\ 
  & \ddots & \vdots & f_g  &  & \vdots \\ 
  &   & -u_{\gamma} &   & \ddots  & \vdots \\
   &  &  &  & & f_g
\end{array}\right]}_{:=F\in\mathbb F^{(\alpha_j+g+1)\times (\nu_{l_0}+\alpha_j+2)}}\begin{pmatrix}
p^{(j)}_0\\ \vdots\\ p^{(j)}_{\nu_{l_0}}\\h^{(j)}_0\\ \vdots\\ h^{(j)}_{\alpha_j}
\end{pmatrix}=\begin{pmatrix}
s_0\\ \vdots\\ s_{\beta_j}\\ 0\\ \vdots\\ 0
\end{pmatrix},$$
where $f_g=1$ because minimal polynomials are monic per definition.
The number of possibilities for $h^{(j)}$ and $p^{(j)}$ to fulfill this equation is at most $t^{-(\nu_{l_0}+\alpha_j+2-\operatorname{rk}(F))}$. Therefore, one has to determine $\operatorname{rk}(F)$. In the following, it is shown that $F$ is of full rank, i.e. $\operatorname{rk}(F)=\min(\alpha_j+g+1,\nu_{l_0}+\alpha_j+2)$.

Case 1: $g\leq \nu_{l_0}+1$\\
In this case, one has to show the surjectivity of $F$. Since $f$ has been defined as the minimal polynomial of $z_0$ and $u(z_0)\neq 0$, one knows that $-u$ and $f$ are coprime. According to Lemma \ref{syl}, this implies that the Sylvester resultant $\operatorname{Res}(-u,f)$, which is the submatrix of $F$ consisting of  columns $1,\hdots,g,\nu_{l_0}+2,\hdots,\nu_{l_0}+\gamma+1$ and rows $1,\hdots, \gamma+g$, is invertible. This is well-defined because $\nu_{l_0}+\alpha_j+2\geq \nu_{l_0}+(\nu_{l_0}+\gamma-g)+2\geq \gamma+g$. Denote by $\tilde{F}$ the matrix for which in $F$ the columns $g+1,\hdots,\nu_{l_0}+1$ are replaced by columns containing only zeros. Obviously, $\operatorname{rk}(\tilde{F})\leq\operatorname{rk}(F)$ and thus, it is sufficient to show the surjectivity of $\tilde{F}$. The span of the first $\gamma+g$ rows of $\tilde{F}$ is equal to the span of the vectors $e_1^\top,\hdots,e_g^\top,e_{\nu_{l_0}+2}^\top,\hdots,e_{\nu_{l_0}+\gamma+1}^\top\in \mathbb F^{1\times (\nu_{l_0}+\alpha_j+2)}$, where $e_j$ denotes the $j$-th unit vector in $\mathbb F^{\nu_{l_0}+\alpha_j+2}$. The matrix consisting of the remaining rows of $\tilde{F}$ has the form $\left[\begin{array}{cccccc}
0 & \hdots & 0 & -1 &   & \ast \\ 
\vdots &   & \vdots &   & \ddots &   \\ 
0 & \hdots & 0 & 0 &   & -1
\end{array}\right]\in\mathbb F^{(\alpha_j+1-\gamma)\times (\nu_{l_0}+\alpha_j+2)}$, i.e. its row span is equal to the span of $e^\top_{\nu_{l_0}+\gamma+2},\hdots,e^\top_{\nu_{l_0}+\alpha_j+2}$. Consequently, the row span of $\tilde{F}$ is equal to the span of $e_1^\top,\hdots,e_g^\top,e_{\nu_{l_0}+2}^\top,\hdots,e_{\nu_{l_0}+\alpha_j+2}^\top$ and hence $\tilde{F}$ and $F$ are surjective.

Case 2: $g>\nu_{l_0}+1$\\
Here, one has to show that $F$ is injective. Choose
$(p^{(j)}_0, \hdots, p^{(j)}_{\nu_{l_0}}, h^{(j)}_0, \hdots, h^{(j)}_{\alpha_j})^\top$ with 
$F\cdot (p^{(j)}_0, \hdots, p^{(j)}_{\nu_{l_0}}, h^{(j)}_0, \hdots, h^{(j)}_{\alpha_j})^\top=(0,\hdots,0)^\top$, i.e.\\
 $-u(z)p^{(j)}(z)+f(z)h^{(j)}(z)=0$.
Since $f$ and $u$ are coprime, it follows that $f$ divides $p^{(j)}$. But because of $\deg(f)=g>\nu_{l_0}+1>\deg(p^{(j)})$ this implies $p^{(j)}\equiv 0$ and hence $h^{(j)}\equiv 0$, too. This shows the injectivity of $F$.\\
In summary, the probability that \eqref{u} is fulfilled is at most
$$g\cdot \varphi_g\cdot t^{g}\cdot \prod_{j=1}^p t^{-(\alpha_j+1)+\min(\alpha_j+g+1,\nu_{l_0}+\alpha_j+2)}=g\cdot \varphi_g\cdot t^{g}\cdot \prod_{j=1}^p t^{\min(g,\nu_{l_0}+1)}.$$
For $g\geq 2$, this probability is $O(t^{2p})=O(t^{p+1})$ since we assumed $\nu_{i}\geq 1$ for $i=1,\hdots,m$ at the beginning of this proof.\\
For $g=1$, write $\mathbb F=\{z_1,\hdots, z_{t^{-1}}\}$ and let $A_i$ be the set of matrices $G\in M(p,n,m)$ for which \eqref{u} is fulfilled for $z_i$. Then, it follows from the preceding computations that $P^{rc}_{p,n,m}(t)=1-\operatorname{Pr}\left(\bigcup_{i=1}^{t^{-1}} A_i\right)+O(t^{p+1})$. Using the inclusion-exclusion principle (see Lemma \ref{ine}), one gets
$$P^{rc}_{p,n,m}(t)=1-\sum_{\emptyset\neq I\subset \{1,\hdots,t^{-1}\}}(-1)^{|I|-1}\operatorname{Pr}(A_I)+O(t^{p+1}).$$
Since $\operatorname{Pr}(A_I):=\operatorname{Pr}(\bigcap_{i\in I}A_i)$ only depends on $|I|$, it follows:
$$P^{rc}_{p,n,m}(t)=1+\sum_{k=1}^{t^{-1}}(-1)^k\binom{t^{-1}}{k}\operatorname{Pr}(\tilde{A}_k)+O(t^{p+1}),$$
where $\operatorname{Pr}(\tilde{A}_k)$ is the probability for the intersection of $k$ pairwisely different sets $A_i$. Furthermore, according to Lemma \ref{res} (a) and (b) with $\tilde{w}:=u$ and $w_j:=-s^{(j)}$, it holds $\operatorname{Pr}(A_i)=t^{p+1}$ for $i=1,\hdots, t^{-1}$ and therefore,
$$P^{rc}_{p,n,m}(t)=1-t^{p}+\sum_{k=2}^{t^{-1}}(-1)^k\binom{t^{-1}}{k}\operatorname{Pr}(\tilde{A}_k)+O(t^{p+1}).$$
Define $a_k(t):=\binom{t^{-1}}{k}\operatorname{Pr}(\tilde{A}_k)\geq 0$. It holds $\operatorname{Pr}(\tilde{A}_{k+1})\leq t\cdot \operatorname{Pr}(\tilde{A}_k)$ since the number of possibilities for $q_{mm}^H$ decreases by (at least) the factor $t$ if one requires an additional zero for this polynomial (for $k+1>n$, there is even no possibility for $q_{mm}^H$), and surely, the number of possibilities for the polynomials from $P$ can only decrease if one has additional conditions. Consequently, the sequence $a_k(t)$ is decreasing and one obtains
$$\sum_{k=2}^{t^{-1}}(-1)^k\binom{t^{-1}}{k}\operatorname{Pr}(\tilde{A}_k)=\sum_{k=2}^{t^{-1}}(-1)^k a_k(t)\leq a_2(t).$$
Hence, it remains to show that $a_2(t)=O(t^{p+1})$. Therefore, one has to consider equations \eqref{u} for $z_0, z_1\in\mathbb F$ with $z_0\neq z_1$ and $u_{l_0,m}(z_0)\neq 0\neq u_{l_1,m}(z_1)$. The number of possibilities for $q_{mm}^H$ is at most $t^{-n+2}$. Moreover, one chooses $l_1=l_0$ if possible. Then, the polynomials $p_{1,l_0},\hdots,p_{p,l_0}$ are fixed at $z_0$ and $z_1$ by the values of the other polynomials from $G$. According to Lemma \ref{res} (c), which could be applied since $\nu_{l_0}\geq 1$, this decreases the number of possibilities by the factor $t^{2p}$. Hence, one has $a_2(t)\leq\binom{t^{-1}}{2}t^{2+2p}\leq t^{2p}\leq t^{p+1}$ in the case $l_1=l_0$. If it is not possible to choose $l_1=l_0$, one knows $u_{l_0,m}(z_1)=0$. Thus, the values of the polynomials $p_{i,l_1}$ for $i=1,\hdots,p$ at $z_1$ are independent of 
the polynomials $p_{i,l_0}$ for $i=1,\hdots,p$. Hence, one first chooses the entries of $P$ but those of columns $l_0$ and $l_1$ randomly, which fixes column $l_1$ at $z_1$. This decreases the number of possibilities by the factor $t^p$. Afterwards, one chooses the polynomials of column $l_1$ in such way that they fulfill the mentioned condition at $z_1$, which finally, fixes the polynomials of column $l_0$ at $z_0$. This contributes again the factor $t^p$ to the probability. In summary, one has $a_2(t)\leq\binom{t^{-1}}{2}t^{2+p+p}\leq  t^{p+1}$, which completes the proof of the whole theorem.
\end{proof}

Inserting the preceding estimation as well as the result from Theorem \ref{THMA} into the formula of Lemma \ref{minl}, one finally obtains an estimation for the probability of minimality.

\begin{theorem}\label{min}\ \\
The probability that a linear discrete-time system described by $(A,B,C,D)\in\mathbb F^{n\times n}\times\mathbb F^{n\times m}\times\mathbb F^{p\times n}\times\mathbb F^{p\times m}$ is minimal is equal to
$$1-t^m-t^p+O(t^{\min(p,m)+1}).$$
\end{theorem}

\section{Probability of Mutual Left Coprimeness}
\label{mlc}

The aim of this section is to calculate the probability that finitely many nonsingular polynomial matrices are mutually left coprime. This main result, stated in Theorem \ref{mut}, will be needed in the following section concerning parallel connections of linear systems. We start with the case $N=2$, then prove a recursion formula for the considered probability, which we will finally solve to achieve Theorem \ref{mut}.

\begin{theorem}\label{2}\ \\
The probability that two matrices $D_1,D_2\in\mathbb F[z]^{m\times m}$ in Hermite form with $\deg(D_i)=n_i$ for $i=1,2$ are left coprime is equal to $1-t^{m}+O(t^{m+1})$.
\end{theorem}

\begin{proof}\ \\
Since the statement is already known for $m=1$ (see Lemma \ref{2coprime}), in the following, it is assumed that $m\geq 2$.
Again we consider the complementary set and show that the cardinality of $S\subset X:=X(n_1,n_2)$ of matrices for which $\mathcal{D}_2$ is not left prime is $O(|X|\cdot t^m)$ and that the cardinality of the subset of $S$ for which $\mathcal{D}_2$ is not of simple form is $O(|X|\cdot t^{m+1})$.\\
Denote the entries of $\mathcal{D}_2$ by $\mathfrak{d}_{ij}$ for $i=1,\hdots,m$ and $j=1,\hdots, 2m$ and choose $z_{\ast}\in\overline{\mathbb F}$ such that $\mathcal{D}_2(z_{\ast})$ is not of full row rank. As in the method of iterated row operations (see Lemma \ref{it} (b)), start considering the first row of this matrix. Either it is identically zero, which means $d_{11}^{(1)}(z_{\ast})=d_{11}^{(2)}(z_{\ast})=0$, or one could assume without restriction that $d_{11}^{(2)}(z_{\ast})\neq 0$. For the first case, one has a cardinality of $|X|\cdot t$ due to the fact that the two polynomials have a common zero. Moreover, they cannot be constant, i.e. $\kappa_m^{(i)}\geq 1$ for $i=1,2$. Thus, it follows from \eqref{codim} that one has an additional factor for the cardinality of at most $t^{2(m-1)}$, which in summary, leads to a cardinality of $O(|X|\cdot t^{2m-1})=O(|X|\cdot t^{m+1})$ for $m\geq 2$ and one is finished. If $d_{11}^{(2)}(z_{\ast})\neq 0$, one proceeds as in the method of iterated row operations, i.e. in the first step, one subtracts multiples of the first row to the rows further down in such way that all entries in column $m+1$ but $d_{11}^{(2)}(z_{\ast})$ are nullified.
From Lemma \ref{it} (b), one knows that there exist $k\in\{1,\hdots,m\}$, a set of column indices $\{j_1,\hdots,j_{k-1}\}\subset\{1,\hdots, 2m\}$ and values $\lambda_{r}\in\mathbb F(z_{\ast})$, which (only) depend on entries $\mathfrak{d}_{ij}$ of $\mathcal{D}_2$ with $j\in\{j_1,\hdots,j_{k-1}\}$ and on $z_{\ast}$, such that
\begin{align}\label{lambdaa}
\mathfrak{d}_{kj}(z_{\ast})=\sum_{r=1}^{k-1}\mathfrak{d}_{rj}(z_{\ast})\cdot\lambda_{r}\qquad\text{for}\ \  j\in\{1,\hdots,2m\}\setminus\{j_1,\hdots,j_{k-1}\}. 
\end{align}
Moreover, since $D_1$ and $D_2$ are lower triangular, it holds $\mathfrak{d}_{ij}\equiv 0$ for $i<j\leq m$ or $i+m<j\leq 2m$. Therefore, \eqref{lambdaa} is equivalent to 
\begin{align}\label{lambdab}
\mathfrak{d}_{kj}(z_{\ast})=\sum_{r=1}^{k-1}\mathfrak{d}_{rj}(z_{\ast})\cdot\lambda_{r}\ \ \text{for}\ j\in\{1,\hdots,k,m+1,\hdots,m+k\}\setminus\{j_1,\hdots,j_{k-1}\}. 
\end{align}
Note that $\{j_1,\hdots,j_{k-1}\}$ is a subset of $\{1,\hdots,k-1,m+1,\hdots,m+k-1\}$ because $\mathfrak{d}_{i,j_i}(z_{\ast})\neq 0$ for $1\leq i\leq k-1$ per construction and hence, $j_i\leq i\leq k-1$ or $m+k-1\geq m+i\geq j_i>m$.
Furthermore, it follows that $d_{kk}^{(1)}(z_{\ast})=\mathfrak{d}_{kk}(z_{\ast})=\sum_{r=1}^{k-1}\mathfrak{d}_{rk}(z_{\ast})\cdot\lambda_{r}=0$ and $d_{kk}^{(2)}(z_{\ast})=\mathfrak{d}_{k,m+k}(z_{\ast})=\sum_{r=1}^{k-1}\mathfrak{d}_{r,m+k}(z_{\ast})\cdot\lambda_{r}=0$.
This could also be seen directly by observing that the iteration process only changes entries beyond the diagonals of the matrices $D_1$ and $D_2$.\\
Thus, one has the conditions $d_{kk}^{(1)}(z_{\ast})=d_{kk}^{(2)}(z_{\ast})=0$, which moreover, ensure $\kappa_{m-k+1}^{(i)}\geq 1$ for $i=1,2$. In particular, this implies that the polynomials $d_{kj}^{(i)}$ for $j<k$ and $i=1,2$ are not fixed to zero by degree restrictions. But since $d_{kj}^{(1)}=\mathfrak{d}_{kj}$ and $d_{kj}^{(2)}=\mathfrak{d}_{k,j+m}$ for $j=1,\hdots,k-1$, one knows from \eqref{lambdab} that $2(k-1)-(k-1)=k-1$ of these polynomials are fixed at $z_{\ast}$ by the remaining polynomials of $\mathcal{D}_2$. We fix $g:=g_{z_{\ast}}$ and apply Lemma \ref{res} (a) and (b) with $\tilde{w}\equiv 1$ and $w_j:=\sum_{r=1}^{k-1}\mathfrak{d}_{rj}\cdot\lambda_r$ for $j\in\{1,\hdots,k-1,m+1,\hdots,m+k-1\}\setminus\{j_1,\hdots,j_{k-1}\}$. One obtains a cardinality that is $O(|X|\cdot\varphi_g\cdot t^{2g+k-1})=O(|X|\cdot t^{g+k-1})$ for the above conditions. Additionally, one gets the factor $t^{2(m-k)}$ from \eqref{codim} since $\kappa_{m-k+1}^{(i)}\geq 1$. In summary, the cardinality is $O(|X|\cdot t^{2m-k+g-1})=O(|X|\cdot t^{m+1})$ for $k\leq m-1$. For $i=m$, one has a factor of $O(|X|\cdot t^m)$. If there is no simple form, it follows from \eqref{codim} that this cardinality is decreased by a factor of at most $t$. Hence, the overall cardinality is $O(|X|\cdot t^{m+1})$. This shows the claim of the first paragraph of this proof for the case that $g$ is fixed. But since $g$ is bounded above by $\min(n_1,n_2)$, it is also valid for summing over all possible values for $g$. Moreover, note that the considered cardinality is $O(|X|\cdot t^{m+g-1})=O(|X|\cdot t^{m+1})$ for $g\geq 2$, even for simple form.\\
It remains to compute the coefficient of $t^m$. It follows from the previous paragraph that for this computation it is sufficient to consider only matrices of the form
$$D_j=\left[\begin{array}{cc}
I_{m-1} & 0 \\ 
d_1^{(j)}\cdots d_{m-1}^{(j)} & d_m^{(j)}
\end{array}\right] $$
for $j=1,2$ for which there exists $z_{\ast}\in\mathbb F$ such that $\left[\begin{array}{cc}
D_1(z_{\ast}) & D_2(z_{\ast}) \end{array}\right]$ is singular, which is the case if and only if $d_m^{(1)}(z_{\ast})=d_m^{(2)}(z_{\ast})=0$ and $d_k^{(1)}(z_{\ast})=d_k^{(2)}(z_{\ast})$ for $1\leq k\leq m-1$. 
According to Lemma \ref{res} (a) and (b) with $\tilde{w}\equiv 1$, $w_j=d_j^{(2)}$ for $j=1,\hdots m-1$ and $g=1$, the probability for this is equal to $\varphi_1\cdot t^{2+m-1}=t^{m}$. Hence, the proof of the theorem is complete.
\end{proof}

In the following, we want to extend the previous result to $N\geq 3$ matrices. But before we approach our actual goal, which is to compute the probability that $N$ matrices are mutually left coprime, we first consider the case of pairwise left coprimeness, which could be deduced from the case $N=2$, where pairwise and mutual left coprimeness coincide.

\begin{theorem}\label{pairwise}\ \\
For $m\geq 1$, the probability of $N$ matrices $D_i\in\mathbb F[z]^{m\times m}$ in Hermite form with $\deg(\det(D_i))=n_i$ for $i=1,\ldots,N$ to be pairwisely left coprime is equal to 
$$1-\frac{N(N-1)}{2}\cdot t^m+O(t^{m+1}).$$
\end{theorem}

\begin{proof}\ \\
Let $S$ be the subset of $X:=X(n_1,\hdots,n_N)$ (see Definition \ref{simple}) for which the tuples consist of pairwisely left coprime matrices and $\mathcal{E}:=\{ij\ |\ 1\leq i<j\leq N\}$. Thus, $S=X\setminus\bigcup_{r\in R}S_r$ with $R=\overline{\mathbb F}\times \mathcal{E}$ and $S_{(z_{\ast},ij)}=\{(D_1,\hdots,D_N)\subset X\ |\ \left[\begin{array}{cc}
D_i(z_{\ast}) & D_j(z_{\ast}) \end{array}\right]\ \text{is singular}\}$.
By the inclusion-exclusion principle, one obtains:
$$|S|=\sum_{T\subset R}(-1)^{|T|}|S_T|\quad \text{with}\quad S_T=\bigcap_{r\in T}S_r\quad\text{and}\quad S_{\emptyset}=X.$$
From Theorem \ref{2}, it follows that the probability $\frac{|S|}{|X|}$ is equal to $1+O(t^m)$. Moreover, from the proof of this theorem, it follows that for the computation of the coefficient of $t^m$, it is sufficient to consider only matrices of the form
$$D_j=\left[\begin{array}{cc}
I_{m-1} & 0 \\ 
d_1^{(j)}\cdots d_{m-1}^{(j)} & d_m^{(j)}
\end{array}\right] $$
for $j=1,\hdots,N$ for which there exist $z_{\ast}\in\mathbb F$ and $1\leq i<j\leq N$ such that $\left[\begin{array}{cc}
D_i(z_{\ast}) & D_j(z_{\ast}) \end{array}\right]$ is singular. Recall that $\left[\begin{array}{cc}
D_i(z_{\ast}) & D_j(z_{\ast}) \end{array}\right]$ is singular if and only if $d_m^{(i)}(z_{\ast})=d_m^{(j)}(z_{\ast})=0$ and $d_k^{(i)}(z_{\ast})=d_k^{(j)}(z_{\ast})$ for $1\leq k\leq m-1$.\\ 
For the case that there exist $\tilde{z}_{\ast}\in\mathbb F$ and $1\leq u<v\leq N$ with $(u,v)\neq (i,j)$ such that 
$\left[\begin{array}{cc}
D_u(\tilde{z}_{\ast}) & D_v(\tilde{z}_{\ast}) \end{array}\right]$ is singular, too, it is obvious that the probability is $O(t^{2m})=O(t^{m+1})$ if $\{i,j\}\cap\{u,v\}=\emptyset$. 
If  $\{i,j\}\cap\{u,v\}\neq\emptyset$, assume without restriction that $j=u$. Then, one could choose $d_1^{(j)},\hdots,d_{m-1}^{(j)}$ as well as $z_{\ast}, \tilde{z}_{\ast}\in\mathbb F$ arbitrarily, which affects that $d_1^{(i)},\hdots,d_{m-1}^{(i)}$ are fixed at $z_{\ast}$ and $d_1^{(v)},\hdots,d_{m-1}^{(v)}$ are fixed at $\tilde{z}_{\ast}$. If $\tilde{z}_{\ast}=z_{\ast}$, it follows from Lemma \ref{res} (a) and (b) that the probability is $t^{-1+3+2(m-1)}=O(t^{m+1})$. If $\tilde{z}_{\ast}\neq z_{\ast}$, for which there are $O(t^{-2})$ possibilities, Lemma \ref{res} (a) and (b) lead to a probability of $O(t^{-2+4+2(m-1)})=O(t^{m+1})$.\\
Thus, in all these cases, the probability is $O(t^{m+1})$, which means that they are not relevant for the coefficient of $t^m$. Consequently, only $T\subset R$ of the form $T=\{(z_{\ast},ij)\}$ with $z_{\ast}\in\mathbb F$ give a contribution to the coefficient of $t^m$, namely $|S_T|=|X|\cdot t^m$ (see end of proof for Theorem \ref{2}). This leads to
\begin{align*}
\frac{|S|}{|X|}&=1-\sum_{z_{\ast}\in\mathbb F,ij\in\mathcal{E}}\frac{|S_{(z_{\ast},ij)}|}{|X|}+O(t^{m+1})=1-\sum_{ij\in\mathcal{E}}t^m+O(t^{m+1})=\\
&=1-\frac{N(N-1)}{2}t^m+O(t^{m+1}).
\end{align*}
\end{proof}

\begin{remark}\ \\
For $m\geq 2$, $N\geq 3$ mutual left coprimeness is a stronger condition than pairwise left coprimeness, as the following example shows.
\end{remark}

\begin{example}\ \\
Consider the pairwisely left coprime matrices in Hermite form $D_1(z)=\left[\begin{array}{cc} 1 & 0 \\ 1 & z \end{array}\right]$, $D_2(z)=\left[\begin{array}{cc}1 & 0 \\ 0 & z\end{array}\right]$ and
$D_3(z)=\left[\begin{array}{cc}z & 0 \\ 0 & 1\end{array}\right]$. We show in two different ways that they are, however, not mutually left coprime. One way to see that is to consider
$$\left[\begin{array}{ccc} D_1(z) & D_2(z) & 0 \\ 0 & D_2(z) & D_3(z) \end{array}\right]= \left[\begin{array}{cccccc}
1 & 0 & 1 & 0 & 0 & 0 \\ 
1 & z & 0 & z & 0 & 0 \\ 
0 & 0 & 1 & 0 & z & 0 \\ 
0 & 0 & 0 & z & 0 & 1
\end{array}\right]. 
$$
Since this matrix is singular for $z=0$, $D_1$, $D_2$ and $D_3$ are not mutually left coprime. A second way to show this is to compute a least common right multiple of e.g. $D_2$ and $D_3$, denoted by $D_{23}$. It is easy to see that one can choose $D_{23}(z)=\left[\begin{array}{cc}z & 0 \\ 0 & z\end{array}\right]$, which is clearly not left coprime with $D_1(z)$. The following lemma shows in particular that for $N=3$ only matrices whose determinants have a common zero (here $z=0$) could be pairwisely but not mutually left coprime.
\end{example}
\begin{lemma}\label{det}\ \\
Let $D_1,\hdots,D_N$ be not mutually left coprime with $\xi\mathcal{D}_N(z_{\ast})=0$ for some $z_{\ast}\in\overline{\mathbb F}$ and $\xi\in\mathbb F(z_{\ast})^{m(N-1)}$. Moreover, every set consisting of $N-1$ of these matrices should be mutually left coprime. Then, it holds $\xi\in(\mathbb F(z_{\ast})\setminus\{0\})^{m(N-1)}$ and $\det(D_i(z_{\ast}))=0$ for $i=1,\hdots,N$.
\end{lemma}

\begin{proof}\ \\
According to Theorem \ref{mutcrit}, $D_1,\hdots,D_N$ are mutually left coprime if and only if\\ $\mathcal{D}_N:=\left[\begin{array}{cccc}
D_1 & D_2 & 0 & 0 \\ 
0 & \ddots & \ddots & 0 \\ 
0 & 0 & D_{N-1} & D_N
\end{array}\right]$
is left prime.\\
Since this is not true, there exist $z_{\ast}\in\overline{\mathbb F}$ and $\xi:=(\xi_1,\hdots,\xi_{N-1})\neq 0$ with $\xi_i\in\mathbb F(z_{\ast})^{1\times m}$ for $i=1,\hdots,N-1$ such that $\xi\mathcal{D}_N(z_{\ast})=0$, i.e.\\ $\xi_1D_1(z_{\ast})=0$, $(\xi_{i-1}+\xi_i)D_i(z_{\ast})=0$ for $i=2,\hdots,N-1$ and $\xi_{N-1}D_N(z_{\ast})=0$. Consequently, one has to show\\
$\xi_1\neq 0$, $\xi_{i-1}+\xi_i\neq 0$ for $i=2,\hdots,N-1$ and $\xi_{N-1}\neq 0$\\
if every proper subset of $\{D_1,\hdots, D_N\}$ consists of mutually left coprime matrices. This is shown per contradiction.\\
If $\xi_{N-1}=0$, i.e. $\tilde{\xi}:=(\xi_1,\hdots,\xi_{N-2})\neq 0$, it follows $\tilde{\xi}\mathcal{D}_{N-1}(z_{\ast})=0$, i.e. $D_1,\hdots,D_{N-1}$ would not be mutually left coprime. Similarly, if $\xi_1=0$,\\
$D_2,\hdots,D_{N}$ would not be mutually left coprime. To show $\xi_{i-1}+\xi_i\neq 0$ for $i=2,\hdots,N-1$, one needs $\xi_i\neq 0$ for $i=2,\hdots,N-2$. If $\xi_k=0$ for some $k=2,\hdots,N-2$, it follows $(\xi_1,\hdots,\xi_{k-1})\neq 0$ or $(\xi_{k+1},\hdots,\xi_{N-1})\neq 0$. In the first case, $D_1,\hdots,D_k$ would not be mutually left coprime, in the second case, $D_{k+1},\hdots,D_N$ would not be mutually left coprime. Now, assume $\xi_{k-1}+\xi_k=0$ for some $k=2,\hdots,N-1$. Define $\hat{\xi}:=(\hat{\xi}_1,\hdots,\hat{\xi}_{N-2})$ with $\hat{\xi}_i=\xi_i$ for $i\leq k-1$ and $\hat{\xi_i}=-\xi_{i+1}$ for $i\geq k$. Then $\hat{\xi}\neq 0$ and $\hat{\xi}_1D_1(z_{\ast})=0$, $(\hat{\xi}_{i-1}+\hat{\xi}_{i})D_{i}(z_{\ast})=0$ for $i=2,\hdots,k-1$, $(\hat{\xi}_{i-1}+\hat{\xi}_{i})(-D_{i+1}(z_{\ast}))=(\xi_{i}+\xi_{i+1})D_{i+1}(z_{\ast})=0$ for $i=k,\hdots, N-2$ and $\hat{\xi}_{N-2}(-D_N(z_{\ast}))=\xi_{N-1}D_N(z_{\ast})=0$. This means that $D_1,\hdots,D_{k-1},-D_{k+1},\hdots,-D_{N}$ are not mutually left coprime. But then $D_1,\hdots,D_{k-1},D_{k+1},\hdots,D_{N}$ are not mutually left coprime, too. 
Consequently, the proof is complete.
\end{proof}

Next, we prove a recursion formula for the probability of mutual left coprimeness, which will be crucial for the proof of Theorem \ref{mut}.

\begin{theorem}\label{rec}\ \\
For $N\geq 2$, the probability that $N$ matrices $D_i\in\mathbb F[z]^{m\times m}$ in Hermite form with $\deg(\det(D_i))=n_i$ for $i=1,\dots, N$ are mutually left coprime is equal to
$$P_m(N)=1+\sum_{k=1}^{N-2}(-1)^k\binom{N}{k}(1-P_m(N-k))-\sum_{i=N-1}^{\min(m,N-1)}t^m+O(t^{m+1}),$$
where $P_m(N-k)$ denotes the probability that $N-k$ such matrices are mutually left coprime.
\end{theorem}

\begin{proof}\ \\
For $N=2$, the formula has already been proven in Theorem \ref{2}. Therefore, one could assume $N\geq 3$. Let $mut(N)$ be the subset of $X(N):=X(n_1,\hdots,n_N)$ for which the tuples consist of mutually left coprime matrices. Moreover, for $i=1,\hdots,N$, $A_i(N)$ should denote the subset of $X(N)$ for which the matrices in the set $\{D_1,\hdots,D_N\}\setminus \{D_i\}$ are not mutually left coprime. Finally, define $A_{N +1}(N):=(X(N)\setminus mut(N))\cap(X(N)\setminus\bigcup_{i=1}^N A_i(N))$, i.e. $A_{N+1}(N)$ consists of those
tuples that are not mutually left coprime but all subsets of $N-1$ matrices
are mutually left coprime.\\
Thus, $mut(N)=X(N)\setminus\bigcup_{i=1}^{N+1}A_i(N)$.
By the inclusion-exclusion principle, one obtains:
\begin{align*}
&|mut(N)|=\sum_{I\subset\{1,\hdots,N+1\}}(-1)^{|I|}|A_I(N)|\\
& \text{with}\quad  A_I(N)=\bigcap_{i\in I}A_i(N)\quad \text{and}\quad  A_{\emptyset}(N)=X(N).
 \end{align*}
From the definition of $A_i(N)$, it follows $A_{N+1}(N)\cap A_i(N)=\emptyset$ for $i=1,\hdots,N$ and consequently, 
\begin{equation}\label{sieb}
|mut(N)|=|X(N)|-|A_{N+1}(N)|+\sum_{\emptyset\neq I\subset\{1,\hdots,N\}}(-1)^{|I|}|A_I(N)|.
\end{equation}
In the following, it is used that $D_1,\hdots,D_N$ are mutually left coprime if and only if\\ $\mathcal{D}_N:=\left[\begin{array}{cccc}
D_1 & D_2 & 0 & 0 \\ 
0 & \ddots & \ddots & 0 \\ 
0 & 0 & D_{N-1} & D_N
\end{array}\right]$
is left prime; see Theorem \ref{mutcrit}.\\
At first, it is shown that $X(N)\setminus mut(N)$ has a cardinality that is $O(|X(N)|\cdot t^m)$ and that the subset of $X(N)\setminus mut(N)$ which contains only tuples of matrices such that $\mathcal{D}_N$ is not of simple form has a cardinality that is $O(|X(N)|\cdot t^{m+1})$. 
Doing this, one uses the following claim.\\
\textbf{Claim 1:}\\
If $D_i=\left[\begin{array}{cc} v_i & 0 \\ w_i & D^{(m-1)}_i \end{array} \right]$ with $v_i\in\mathbb F[z]$ and $v_i(z_{\ast})\neq 0$ for $i=1,\hdots,N$ and some $z_{\ast}\in\overline{\mathbb F}$, it holds
\begin{align*}
&\operatorname{rk}(\mathcal{D}_N(z_{\ast}))=\\ &=\operatorname{rk}\left(\left[\begin{array}{ccccc}
w_1-\frac{v_1}{v_2}w_2 & D^{(m-1)}_1 & D^{(m-1)}_2 &   &   \\ 
(\frac{w_3}{v_3}-\frac{w_2}{v_2})v_1 &   &  D^{(m-1)}_2 &   D^{(m-1)}_3 &     \\ 
\vdots  & & \ddots & \ddots &  \\ 
(-1)^N(\frac{w_{N-1}}{v_{N-1}}-\frac{w_N}{v_N})v_1 &   &    &  D^{(m-1)}_{N-1}  &  D^{(m-1)}_N
\end{array}\right](z_{\ast})\right)+\\
&+N-1.
\end{align*}
Proof of claim 1:\\
This proof uses a method that is similar to the method of iterated row operations introduced in Lemma \ref{it} (b). However, one applies only one iteration step to special subblocks of $\mathcal{D}_N(z_{\ast})$.\\
One starts adding row $(N-2)m+1$ of $\mathcal{D}_N(z_{\ast})$ times $\frac{-w_{N,r}}{v_N}(z_{\ast})$ to row $(N-2)m+1+r$  for  $r=1,\hdots,m-1$. Here, $w_{N,r}$ denotes the $r$-th component of the vector $w_N$. Afterwards, deleting row $(N-2)m+1$ and column $(N-1)m+1$ decreases the rank of $\mathcal{D}_N(z_{\ast})$ by $1$.
This affects only the block $[D_{N-1}\ D_N](z_{\ast})$, whose first row and $(m+1)$-th column are deleted and whose first column is changed to $(w_{N-1}-\frac{v_{N-1}}{v_N}w_N)(z_{\ast})$. Moreover, some zeros of the zero blocks are deleted. Now, one continues adding multiples of row $(N-3)m+1$ to all rows further down in such way that the entries in these rows which are in column $(N-2)m+1$ are nullified. This additionally changes the entries of these rows that are in column $(N-3)m+1$ but no other entries. Afterwards, row $(N-3)m+1$ and column $(N-2)m+1$ are deleted decreasing the rank by 1. Hence, per induction with respect to $N$, one could assume
\begin{align*}
&\operatorname{rk}(\mathcal{D}_N(z_{\ast}))=\\
&\operatorname{rk}\left(\left[\begin{array}{cccccc}
v_1 & 0 & v_2 & 0 & &\\
w_1 & D^{(m-1)}_1 &   w_2 & D^{(m-1)}_2  & &\\
& & w_2-\frac{v_2}{v_3}w_3 &   D^{(m-1)}_2 &   D^{(m-1)}_3 &     \\ 
& & \vdots & \ddots & \ddots &  \\ 
& & (-1)^{N-1}(\frac{w_{N-1}}{v_{N-1}}-\frac{w_N}{v_N})v_2 &   &  D^{(m-1)}_{N-1}  &  D^{(m-1)}_N
\end{array}\right](z_{\ast})\right)\\
&+N-2.
\end{align*}
Now, one adds the first row to all rows beyond in such way that column $m+1$ is nullified and afterwards, deletes the first row and $(m+1)$-th column. Doing this, one gets 
\begin{align*}
&\operatorname{rk}(\mathcal{D}_N(z_{\ast}))=\\
&\operatorname{rk}\left(\left[\begin{array}{ccccc}
w_1-\frac{v_1}{v_2}w_2 & D^{(m-1)}_1 & D^{(m-1)}_2 &   &   \\ 
(\frac{w_3}{v_3}-\frac{w_2}{v_2})v_1 &   &  D^{(m-1)}_2 &   D^{(m-1)}_3 &     \\ 
\vdots  & & \ddots & \ddots &  \\ 
(-1)^N(\frac{w_{N-1}}{v_{N-1}}-\frac{w_N}{v_N})v_1 &   &    &  D^{(m-1)}_{N-1}  &  D^{(m-1)}_N
\end{array}\right](z_{\ast})\right)\\
&+N-1
\end{align*}
and claim 1 is proven.\\
Next, denote by $\mathcal{D}_N^{(\tilde{m})}$ the matrix formed by blocks which consist of the last $\tilde{m}$ columns and rows of the matrices $D_i$ for $i=1,\hdots,N$ and define the sets
\begin{align*}
A(\tilde{m},k)&:=\{\mathcal{D}_N \text{ with}\ \operatorname{rk}(\mathcal{D}_N^{(\tilde{m})})(z_{\ast})\leq (N-1)\tilde{m}-k\ \text{for some $z_{\ast}\in\overline{\mathbb F}$}\} \ \text{and}\\
A^f(\tilde{m},k)&:=A(\tilde{m},k)\cap\{\mathcal{D}_N\ \text{with no simple form}\}.
\end{align*}
\textbf{Claim 2:}\\
For $\tilde{m},k\in\mathbb N$ with $\tilde{m}+k\leq m+1$, the cardinality of $A(\tilde{m},k)$ is\\
$O(|X(N)|\cdot t^{\tilde{m}+k-1})$ and the cardinality of $A^f(\tilde{m},k)$ is $O(|X(N)|\cdot t^{\tilde{m}+k})$.\\
Proof of claim 2:\\
We proceed per induction with respect to $\tilde{m}$ and start with the base clause $\tilde{m}=1$. For $k>N-1$, it holds $A(1,k)=A^f(1,k)=\emptyset$, which has cardinality $O(|(X(N)|\cdot t^{m(n_1+\cdots+n_N)})=O(|X(N)|\cdot t^{\tilde{m}+k+1})$ since $m(n_1+\cdots+n_N)\geq mN\geq (\tilde{m}+k-1)\cdot 2=2k\geq k+2= \tilde{m}+k+1$ because $k\geq N\geq 2$. Thus, it is sufficient to consider the case $k\leq N-1$. The blocks that form $\mathcal{D}_N^{(1)}$ are just the scalar polynomials $d^{(i)}:=d_{m,m}^{(i)}$ for $i=1,\hdots,N$. In $A(1,k)$, there exists $z_{\ast}\in\overline{\mathbb F}$ such that all matrices consisting of $N-k$ rows of $\mathcal{D}_N^{(1)}$ are not of full rank at $z_{\ast}$. Especially, the first $N-k$ rows are linearly dependent at $z_{\ast}$, which is equivalent to the fact that at least two of the polynomials $d^{(1)},\hdots,d^{(N-k+1)}$ are zero at $z_{\ast}$. Since permutating the set $\{d^{(1)},\hdots,d^{(N)}\}$ does not change the rank of $\mathcal{D}_N^{(1)}$, every subset of $N-k+1$ polynomials contains two polynomials that are zero at $z_{\ast}$.\\
Now, one proceeds in the following way: First, choose the polynomials\\
$d^{(1)},\hdots,d^{(N-k+1)}$ and denote the polynomials that are zero at $z_{\ast}$ by $d_{1,1}$ and $d_{1,2}$. Then, consider the set of polynomials $\{d^{(1)},\hdots,d^{(N-k+2)}\}\setminus\{d_{1,1}\}$ and iterate this procedure until ending up with the set\\
$\{d^{(1)},\hdots,d^{(N)}\}\setminus\{d_{1,1},\hdots, d_{k-1,1}\}$. In summary, at least the $k+1$ different polynomials $d_{1,1},\hdots, d_{k-1,1},d_{k,1},d_{k,2}$ are zero at $z_{\ast}$. Hence, the cardinality of $A(1,k)$ is $O(|X(N)|\cdot t^{1+k-1})$ (see Lemma \ref{2coprime}).
Moreover, the cardinality of $A^f(1,k)$ is $O(|X(N)|\cdot t^{k+1})$ since the not simple form decreases the cardinality by at least the factor $t$; see \eqref{codim}.\\
For the step from $\tilde{m}$ to $\tilde{m}+1$, three cases are distinguished.\\
Case 1: $k>(N-1)(\tilde{m}+1)$\\
Here, $A(\tilde{m}+1,k)=A^f(\tilde{m}+1,k)=\emptyset$, which has a cardinality that is $O(|X(N)|\cdot t^{m(n_1+\cdots+n_N)})=O(|X(N)|\cdot t^{\tilde{m}+k+1})$ since $m(n_1+\cdots+n_N)\geq mN\geq 2m\geq m+1\geq \tilde{m}+k+1$.\vspace{2mm}\\
Case 2: $k=(N-1)(\tilde{m}+1)$\\
This means $\operatorname{rk}(\mathcal{D}^{(\tilde{m}+1)}_N(z_{\ast}))=0$, i.e. $\mathcal{D}^{(\tilde{m}+1)}_N(z_{\ast})\equiv 0$. In particular, all diagonal elements are identically zero, which implies $\kappa_j^{(i)}\geq 1$ for $j=1,\hdots,\tilde{m}+1$ and $i=1,\hdots N$. Consequently, according to Lemma \ref{2coprime} and equation \eqref{codim}, the corresponding cardinality is $O(|X(N)|\cdot t^{N(\tilde{m}+1)-1+N\tilde{m}})=O(|X(N)|\cdot t^{k+\tilde{m}+1})$.
Case 3: $k\leq (N-1)(\tilde{m}+1)-1\Leftrightarrow \tilde{m}(N-1)\geq k+2-N$\\
Case 3 is divided into three subcases.\\
Case 3.1: $d_{m-\tilde{m}, m-\tilde{m}}^{(1)}(z_{\ast})=\cdots=d_{m-\tilde{m},m-\tilde{m}}^{(N)}(z_{\ast})=0$\\
That these polynomials have a common zero contributes a factor of $O(t^{N-1})$ to the cardinality. Additionally, these polynomials cannot be identically $1$, which implies $1\leq \kappa^{(i)}_{m-(m-\tilde{m})+1}=\kappa^{(i)}_{\tilde{m}+1}$ (in particular, one has no simple form) and contributes a factor of $O(t^{N\tilde{m}})$ to the cardinality; see \eqref{codim}. In summary, the cardinality is $O(|X(N)|\cdot t^{\tilde{m}+1+k})$ since $N-1+N\tilde{m}\geq N-1+\tilde{m}+k+2-N=\tilde{m}+1+k$.\\
Case 3.2: $d_{m-\tilde{m}, m-\tilde{m}}^{(1)}(z_{\ast}),\hdots,d_{m-\tilde{m},m-\tilde{m}}^{(l-1)}(z_{\ast})\neq 0$ and $d_{m-\tilde{m}, m-\tilde{m}}^{(l)}(z_{\ast})=\cdots=d_{m-\tilde{m},m-\tilde{m}}^{(N)}(z_{\ast})=0$ for some $l\in\{2,\hdots,N\}$\\
All entries of row $(l-2)(\tilde{m}+1)+1$ of $\mathcal{D}_N^{(\tilde{m}+1)}(z_{\ast})$ but $d_{m-\tilde{m},m-\tilde{m}}^{(l-1)}(z_{\ast})\neq 0$ are equal to zero. Hence, deleting the row and column of this entry decreases the rank by $1$. After that, for $l\geq 3$, row $(l-3)(\tilde{m}+1)+1$ of the remaining matrix consists only of zeros but $d_{m-\tilde{m},m-\tilde{m}}^{(l-2)}(z_{\ast})\neq 0$ and the procedure could be iterated until all rows and columns of the entries $d_{m-\tilde{m}, m-\tilde{m}}^{(1)}(z_{\ast}),\hdots,d_{m-\tilde{m},m-\tilde{m}}^{(l-1)}(z_{\ast})$ are deleted and the rank is decreased by $l-1$. Moreover, the entries of the rows $(l-1+j)(\tilde{m}+1)+1$ for $j=0,\hdots, N-l-1$ of $\mathcal{D}_N^{(\tilde{m}+1)}$ that are no fixed zeros are contained in the set $\{d_{m-\tilde{m}, m-\tilde{m}}^{(l)},\hdots, d_{m-\tilde{m},m-\tilde{m}}^{(N)}\}$. Thus, these rows consist only of zeros at $z_{\ast}$ and could be deleted without changing the rank. Deleting also the columns of these entries, could only decrease the rank and one ends up with $\mathcal{D}_N^{(\tilde{m})}(z_{\ast})$, which, consequently, has rank at most $(N-1)(\tilde{m}+1)-k-l+1=(N-1)\tilde{m}-(k+l-N)$. Per induction, this leads to a cardinality of $O(|X(N)|\cdot t^{\tilde{m}+k+l-N-1})$.\\
Since each $z_{\ast}$ such that $\mathcal{D}_N^{(\tilde{m})}(z_{\ast})$ is not of full row rank is a common zero of all full size subminors of $\mathcal{D}_N^{(\tilde{m})}$, it is, in particular, a zero of $\prod_{i=1}^N\det(D_i^{(\tilde{m})})$. Therefore, for each $\mathcal{D}_N^{(\tilde{m})}$, there exist only finitely many such $z_{\ast}\in\overline{\mathbb F}$. Consequently, one could regard $z_{\ast}$ as already fixed when considering the further conditions $d_{m-\tilde{m}, m-\tilde{m}}^{(l)}(z_{\ast})=\cdots=d_{m-\tilde{m},m-\tilde{m}}^{(N)}(z_{\ast})=0$. Thus, these conditions contribute the factor $t^{N-l+1+\tilde{m}(N-l+1)}$ to the cardinality. Here, the summand $\tilde{m}(N-l+1)$ is due to the fact that $d_{m-\tilde{m},m-\tilde{m}}^{(i)}\not\equiv 1$ for $i=l,\hdots,N$; see \eqref{codim}. In summary, the cardinality is $O(|X(N)|\cdot t^{\tilde{m}+k+\tilde{m}(N-l+1)})=O(|X(N)|\cdot t^{\tilde{m}+k+1})$.\\
Case 3.3: $d_{\tilde{m}-m,\tilde{m}-m}^{(i)}(z_{\ast})\neq 0$ for $i=1,\hdots,N$\\
From claim 1 with $D_i^{(\tilde{m}+1)}=\left[\begin{array}{cc}
v_i & 0 \\ 
w_i & D_i^{(\tilde{m})}
\end{array}\right]$ and $v_i=d_{\tilde{m}-m,\tilde{m}-m}^{(i)}$ for $i=1,\hdots,N$, one knows $\operatorname{rk}(\mathcal{D}^{(\tilde{m}+1)}_N(z_{\ast}))=\operatorname{rk}(r(z_{\ast})\ \mathcal{D}_N^{(\tilde{m})}(z_{\ast}))+N-1$, where $r=\begin{pmatrix}w_1-\frac{v_1}{v_2}w_2\\ \hdots\\ (-1)^N(\frac{w_{N-1}}{v_{N-1}}-\frac{w_N}{v_N})v_1\end{pmatrix}\in\mathbb F^{(N-1)\tilde{m}}[z]$. It follows $\operatorname{rk}(r(z_{\ast})\ \mathcal{D}_N^{(\tilde{m})}(z_{\ast}))\leq (N-1)\tilde{m}-k$. If $\operatorname{rk}(\mathcal{D}_N^{(\tilde{m})}(z_{\ast}))\leq (N-1)\tilde{m}-k-1$, one knows per induction that the cardinality is $O(|X(N)|\cdot t^{\tilde{m}+k})$. If $\mathcal{D}_N$ is not of simple form, one has an additional factor of at most $t$, no matter if $\mathcal{D}_N^{(\tilde{m})}$ is of simple form or not.\\
If $\operatorname{rk}(\mathcal{D}_N^{(\tilde{m})}(z_{\ast}))=(N-1)\tilde{m}-k$, one knows that the cardinality is $O(|X(N)|\cdot t^{\tilde{m}+k-1})$ and additionally, that $r(z_{\ast})$ lies in the column span of $\mathcal{D}_N^{(\tilde{m})}(z_{\ast})$. Hence, one has to show that this second condition leads to an additional factor for the probability that is $O(t)$.
As seen above, 
for each $\mathcal{D}_N^{(\tilde{m})}$, there exist only finitely many $z_{\ast}\in\overline{\mathbb F}$ with $\operatorname{rk}(\mathcal{D}_N^{(\tilde{m})}(z_{\ast}))=(N-1)\tilde{m}-k$. Consequently, one just has to consider the case that $z_{\ast}$ and $\mathcal{D}_N^{(\tilde{m})}$ are fixed and the vector $r(z_{\ast})$ lies in the column span of $\mathcal{D}_N^{(\tilde{m})}(z_{\ast})$.
If $(N-1)\tilde{m}-k<0$, one has a cardinality that is $O(|X(N)|\cdot t^{\tilde{m}+k+1})$, anyway (see case 1).\\
If $(N-1)\tilde{m}-k=0$, one has $\mathcal{D}_N^{(\tilde{m})}(z_{\ast})\equiv 0$ and $r(z_{\ast})\equiv 0$.
Hence, $\kappa_j^{(i)}\geq 1$ for $i=1, \hdots, N$ and $j=1,\hdots,\tilde{m}$, and thus, the vectors $w_1, \hdots, w_N$ contain no entries that are fixed (to zero) by degree conditions. Furthermore, one has, amongst others, $w_1(z_{\ast})=\frac{v_1}{v_2}w_2(z_{\ast})$, which means, in particular, that the first component of $w_1$ is fixed by the other polynomials, which contributes a factor that is $O(t)$ to the cardinality; see Lemma \ref{res} (b)).\\
If $(N-1)\tilde{m}-k>0$, one could choose $(N-1)\tilde{m}-k$ linearly independent rows in $\mathcal{D}_N^{(\tilde{m})}(z_{\ast})$. If there exist $i\in\{1,\hdots,N-1\}$ and $j\in\{1,\hdots,\tilde{m}\}$ such that the $j$-th components of $w_i$ and $w_{i+1}$ are fixed to zero by degree conditions, which is the case if and only if $\kappa_{\tilde{m}+1-j}^{(i)}=\kappa_{\tilde{m}+1-j}^{(i+1)}=0$, row $\tilde{m}(i-1)+j$ has ones in the positions $\tilde{m}(i-1)+j$ and $\tilde{m}i+j$ and zeros, elsewhere. Thus, all these rows are linearly independent and one could assume without restriction that they are contained in the chosen set of linearly independent rows. Permute the rows of $[r\ \mathcal{D}_N^{(\tilde{m})}]$ in such way that the entries of the chosen rows of $\mathcal{D}_N^{(\tilde{m})}$ are contained in rows $1, \hdots, (N-1)\tilde{m}-k$, which we call upper part, while the other rows of $\mathcal{D}_N^{(\tilde{m})}$ should be called lower part. Clearly, interchanging rows does not change the rank of the whole matrix. In the following, $[r\ \mathcal{D}_N^{(\tilde{m})}]$ should denote the matrix with the already interchanged rows. Note that the Hermite form is lost by this interchanging process but that does not matter for the following considerations.\\
Next, delete $\tilde{m}+k$ columns of $\mathcal{D}_N^{(\tilde{m})}(z_{\ast})$ such that the remaining entries of the upper part form an invertible matrix, denoted by $\overline{D}$. The matrix consisting of the remaining entries of the lower part should be denoted by $\underline{D}$. Analogously, denote the corresponding parts of $r$ by $\overline{r}\in\mathbb F [z]^{(N-1)\tilde{m}-k}$ and $\underline{r}\in\mathbb F [z]^k$, respectively. Since the column rank of $\begin{pmatrix} \overline{D}(z_{\ast})\\ \underline{D}(z_{\ast}) \end{pmatrix}$ is still $(N-1)\tilde{m}-k$, its column span is equal to the column span of $\mathcal{D}_N^{(\tilde{m})}(z_{\ast})$ and therefore, $r(z_{\ast})$ is contained in it. Hence, there exists $\lambda\in\overline{\mathbb F}^{(N-1)\tilde{m}-k}$ with $\begin{pmatrix} \overline{D}(z_{\ast})\lambda\\ \underline{D}(z_{\ast})\lambda \end{pmatrix}=\begin{pmatrix} \overline{r}(z_{\ast})\\ \underline{r}(z_{\ast}) \end{pmatrix}$, i.e. $\underline{r}(z_{\ast})=\underline{D}(z_{\ast})\overline{D}^{-1}(z_{\ast})\overline{r}(z_{\ast})$. Denote the last row of $\underline{D}\overline{D}^{-1}$ by $d_1,\hdots d_{(N-1)\tilde{m}-k}$. Then $\underline{r}_k(z_{\ast})=\sum_{l=1}^{(N-1)\tilde{m}-k}d_l\overline{r}_l(z_{\ast})$. Moreover, $w_{i,j}$ should denote the $j$-th component of the vector $w_i\in\mathbb F[z]^{\tilde{m}}$. Thus, $\underline{r}_k=(-1)^i\left(\frac{w_{i,j}}{v_i}-\frac{w_{i+1,j}}{v_{i+1}}\right)v_1$ for some $i\in\{1,\hdots,N-1\}$ and $j\in\{1,\hdots,\tilde{m}\}$. 
The polynomials $w_{i,j}$ and $w_{i,j+1}$ could not both be fixed to zero due to degree conditions since otherwise $(-1)^i\left(\frac{w_{i,j}}{v_i}-\frac{w_{i+1,j}}{v_{i+1}}\right)v_1$ would belong to $\overline{r}$ per construction of upper and lower part. Assume without restriction that $w_{i,j}$ is no fixed zero. \\
First, consider the case that $(-1)^i\left(\frac{w_{i,j}}{v_i}-\frac{w_{i-1,j}}{v_{i-1}}\right)v_1$ is not contained in $\overline{r}$ and hence, $w_{i,j}$ is not contained in the term for any entry of $\overline{r}$. Then, one could choose all polynomial entries of $\mathcal{D}_N^{(\tilde{m}+1)}$ but $w_{i,j}$ arbitrarily, which effects that $w_{i,j}(z_{\ast})$ is fixed. However, this contributes a factor of $O(t)$ to the cardinality; see Lemma \ref{res} (b). If $(-1)^i\left(\frac{w_{i,j}}{v_i}-\frac{w_{i-1,j}}{v_{i-1}}\right)v_1$ is contained in $\overline{r}$, assume without restriction that it equals $\overline{r}_1$. Then, one has 
\begin{align}\label{wij}
\frac{w_{i,j}}{v_i}(1-d_1)v_1(z_{\ast})=\left(\frac{w_{i+1,j}}{v_{i+1}}-\frac{w_{i-1,j}}{v_{i-1}}d_1\right)v_1(z_{\ast})+(-1)^i\sum_{l=2}^{(N-1)\tilde{m}-k}d_l\overline{r}_l(z_{\ast}).
\end{align}
Consider $d_1(z_{\ast})=\sum_{l=1}^{(N-1)\tilde{m}-k}\underline{D}_{k,l}\overline{D}^{-1}_{l,1}(z_{\ast})$.\\
Case 3.3.1: The entries of $\kappa$ are so that $d_1\equiv 0$ (by degree conditions).\\
Here, one has, in particular, $d_1(z_{\ast})\neq 1$ and could, therefore, solve equation \eqref{wij} with respect to $w_{i,j}(z_{\ast})$. Hence, one has a factor that is $O(t)$ for the cardinality and is done.\\
Case 3.3.2: The entries of $\kappa$ are not so that they imply $d_1\equiv 0$.\\
If $d_1(z_{\ast})=0$, which also implies that one could solve equation \eqref{wij} with respect to $w_{i,j}(z_{\ast})$ and consequently, is done as well, there exists $l_{\ast}$ such that neither $\underline{D}_{k.l_{\ast}}\equiv 0$ nor $\overline{D}^{-1}_{l_{\ast},1}\equiv 0$ due to degree restrictions (caused by the values of $\kappa$). Thus, either $\overline{D}^{-1}_{l_{\ast},1}(z_{\ast})=0$, which leads to a factor which is $O(t)$ for the cardinality, or one could solve the equation $d_1(z_{\ast})=0$ with respect to $\underline{D}_{k,l_{\ast}}(z_{\ast})$ (that is no fixed $1$ per construction of upper and lower part), which provides the factor $O(t)$, too. Therefore, the probability that $d_1(z_{\ast})=0$ if not $d_1\equiv 0$ due to degree conditions, is $O(t)$. Hence, it only remains to investigate what happens if $d_1(z_{\ast})\neq 0$, which is true with a probability of $1-O(t)$ in the considered case. This implies that the probability that $\operatorname{rk}(\mathcal{D}_N^{(\tilde{m})}(z_{\ast}))=(N-1)\tilde{m}-k$ under the condition $d_1(z_{\ast})\neq 0$ is $\frac{O(t^{\tilde{m}+k-1})}{1-O(t)}=O(t^{\tilde{m}+k-1})$.\\
Per construction of $\overline{D}$ and $\underline{D}$, it does not influence the condition $\operatorname{rk}(\mathcal{D}_N^{(\tilde{m})}(z_{\ast}))=(N-1)\tilde{m}-k$, which nonzero value is taken by $d_1(z_{\ast})$. This is true since $\overline{D}(z_{\ast})$ is invertible and therefore, the rows of $\underline{D}(z_{\ast})$ are linearly dependent on the rows of $\overline{D}(z_{\ast})$, anyway. Moreover, multiplying a row by a nonzero factor, does not influence linear dependence, i.e. does not influence the number of possibilities for the entries of $\mathcal{D}_N^{(\tilde{m})}$ which are not contained in $\overline{D}$ or $\underline{D}$. 
If $d_1(z_{\ast})\neq 1$, one could solve equation \eqref{wij} with respect to $w_{i,j}(z_{\ast})$ and is done.\\
If $1=d_1(z_{\ast})=\sum_{l=1}^{(N-1)\tilde{m}-k}\underline{D}_{k,l}\overline{D}^{-1}_{l,1}(z_{\ast})$, there exists $l_0\in\{1,\hdots, (N-1)\tilde{m}-k\}$ such that $\overline{D}^{-1}_{l_0,1}(z_{\ast})\neq 0$ and $\underline{D}_{k,l_0}$ is no fixed zero (it cannot be a fixed $1$ per construction of upper and lower part). Consequently,  one could solve the above equation with respect to $\underline{D}_{k,l_0}(z_{\ast})$. Because it follows from the preceding considerations that the condition $d_1(z_{\ast})=1$ is independent from the condition $\operatorname{rk}(\mathcal{D}_N^{(\tilde{m})}(z_{\ast}))=(N-1)\tilde{m}-k$, one gets an additional factor that is $O(t)$ for the cardinality. 
As in previous cases, the cardinality is decreased by a factor of at most $t$ if one has no simple form and thus, all cases are finished.\\
Note that it is sufficient to consider these three cases since the order of $D_1,\hdots,D_N$ is not relevant for the property to be mutually left coprime. Therefore, the proof of claim 2 is complete.\\
Using claim 2 with $\tilde{m}=m$ and $k=1$, completes the first part of this proof.\\
Next, one needs to compute the probability for the case that $\mathcal{D}_N\subset A_{N+1}(N)$ is of simple form, i.e. the case that $D_i=\left[\begin{array}{cc}
I_{m-1} & 0 \\ 
d_1^{(i)}\cdots d_{m-1}^{(i)} & d_m^{(i)}
\end{array}\right]$ for $i=1,\hdots,N$.\\
\textbf{Claim 3}:\\
For 
\begin{align*}
&\mathcal{D}:=\\
&\left[\begin{array}{ccccccc}
d^{(1)}_1-d^{(2)}_1 & \hdots & d^{(1)}_{m-1}-d^{(2)}_{m-1} & d^{(1)}_m & d^{(2)}_m &   &   \\ 
d^{(3)}_1-d^{(2)}_1 & \hdots & d^{(3)}_{m-1}-d^{(2)}_{m-1} &  &  d^{(2)}_m &   d^{(3)}_m &     \\ 
\vdots & & \vdots & & \ddots & \ddots &  \\ 
(-1)^N(d^{(N-1)}_1-d^{(N)}_1) & \hdots & (-1)^N(d^{(N-1)}_{m-1}-d^{(N)}_{m-1}) &   &    &  d^{(N-1)}_m  &  d^{(N-1)}_m
\end{array}\right],
\end{align*}
it holds
$\operatorname{rk}(\mathcal{D}_N)=\operatorname{rk}(\mathcal{D})+(m-1)(N-1)$.\\
Proof of claim 3:\\
One proceeds as in the proof of claim 1 (with $v_i=1$) and achieves:
\begin{align*}
&\operatorname{rk}(\mathcal{D}_N)=\\
&=\operatorname{rk}\left[\begin{array}{ccccc}
w_1-w_2 & D^{(m-1)}_1 & D^{(m-1)}_2 &   &   \\ 
w_3-w_2 &   &  D^{(m-1)}_2 &   D^{(m-1)}_3 &     \\ 
\vdots  & & \ddots & \ddots &  \\ 
(-1)^N(w_{N-1}-w_N) &   &    &  D^{(m-1)}_{N-1}  &  D^{(m-1)}_N
\end{array}\right]+N-1,
\end{align*}
where $w_i=(0,\hdots,0, d_1^{(i)})^\top\in\mathbb F[z]^{m-1}$ and $D_i^{(m-1)}\in\mathbb F[z]^{(m-1)\times (m-1)}$ is in simple form for $i=1,\hdots, N$. One iterates this procedure $m-1$ times and since one always adds the first row of a block to rows further down, the first column of the whole matrix is not affected. Deleting the corresponding row, only deletes one zero in each of the vectors $w_i$. After $m-1$ iterations, one ends up with the statement of claim 3 and thus, claim 3 is proven.\\
Consequently, for simple form, $D_1,\hdots, D_N$ are not mutually left coprime if and only if there exist $z_{\ast}\in\overline{\mathbb F}$ and $\xi\in\overline{\mathbb F}^{1 \times (N-1)}\setminus\{0\}$ such that $\xi\mathcal{D}(z_{\ast})=0$,
which is equivalent to
\begin{align}\label{above}
\xi_1d_i^{(1)}(z_{\ast})-(\xi_1+\xi_2)d_i^{(2)}(z_{\ast})+\cdots
&+(-1)^N(\xi_{N-2}+\xi_{N-1})d_i^{(N-1)}(z_{\ast})+\nonumber\\
+(-1)^{N+1}\xi_{N-1}d_i^{(N)}(z_{\ast})&=0\ \text{for}\ i=1,\hdots m-1\nonumber\\
\xi_1d_m^{(1)}(z_{\ast})&=0\nonumber\\
(\xi_{i-1}+\xi_i)d_m^{(i)}(z_{\ast})&=0\ \text{for}\ i=2,\hdots,N-1\nonumber\\ \xi_{N-1}d_m^{(N)}(z_{\ast})&=0.
\end{align}
Next, define $\tilde{A}_{N+1}(N)$ as the subset of $X(N)$ for which there exists $z_{\ast}\in\overline{\mathbb F}$ such that $\mathcal{D}_N(z_{\ast})$ is singular and $\det(D_i(z_{\ast}))=0$ for $i=1,\hdots,N$. From Lemma \ref{det}, it follows $A_{N+1}(N)\subset\tilde{A}_{N+1}(N)$. In the following, we compute the cardinality of $\tilde{A}_{N+1}(N)$ for simple form, i.e. the probability that there exist $z_{\ast}\in\overline{\mathbb F}$ and $\xi\in\mathbb F(z_{\ast})^{1 \times (N-1)}\setminus\{0\}$ with $d_m^{(i)}(z_{\ast})=0$ for $i=1,\hdots,N$, such that the first $m-1$ equations of \eqref{above} are fulfilled. Firstly, that these $N$ polynomials have a common zero gives a factor to the cardinality of $O(t^{N-1})$.\\
If $m<N-1$, this leads to a cardinality that is $O(|X(N)|\cdot t^{m+1})$.\\ 
To consider the case $m\geq N-1$, one sorts the possible values for $z_{\ast}$ with respect to the degree of their minimal polynomial and sets $g:=g_{z_{\ast}}$. Then, $\xi\in(\mathbb F^g)^{1\times(N-1)}\setminus \{0\}$, where $\mathbb F^g$ denotes the extension field of $\mathbb F$ with $t^{-g}$ elements. Hence, there are $t^{-g(N-1)}-1$ possibilities for the choice of $\xi$. 
Since there exists $i\in\{1,\hdots,N-1\}$ with $\xi_i\neq 0$, at least one element of $\{\xi_1,\xi_1+\xi_2,\hdots,\xi_{N-2}+\xi_{N-1},\xi_{N-1}\}$ is unequal to zero and thus, there exists $j_0\in\{1,\hdots,N\}$ such that one could solve equations $1$ to $m-1$ of \eqref{above} with respect to $d_i^{(j_0)}(z_{\ast})$ for $i=1,\hdots, m-1$. Assume without restriction that $j_0=1$. If the other entries of $\mathcal{D}_N$ as well as $z_{\ast}$ are fixed, for $\xi,\hat{\xi}\in(\mathbb F^g)^{1\times(N-1)}\setminus \{0\}$, one obtains the same values for $\xi_1 d_i^{(1)}(z_{\ast})$ if and only if 
\begin{align*}
&(\xi_1+\xi_2)d_i^{(2)}(z_{\ast})+\cdots+(-1)^{N+1}\xi_{N-1}d_i^{(N)}(z_{\ast})=\\
&(\hat{\xi}_1+\hat{\xi}_2)d_i^{(2)}(z_{\ast})+\cdots+(-1)^{N+1}\hat{\xi}_{N-1}d_i^{(N)}(z_{\ast})
\end{align*}
for $i=1,\hdots,m-1$. But these equations hold if and only if the vector $\begin{pmatrix}\xi_1+\xi_2\\ \vdots\\ (-1)^{N-1}\xi_{N-1} \end{pmatrix}-\begin{pmatrix}\hat{\xi}_1+\hat{\xi}_2\\ \vdots\\ (-1)^{N-1}\hat{\xi}_{N-1} \end{pmatrix}$ is contained in the kernel of 
$$D(z_{\ast}):=\left[\begin{array}{ccc}
d_1^{(2)} & \hdots & d_1^{(N)} \\ 
\vdots &  & \vdots \\ 
d_{m-1}^{(2)} & \hdots & d_{m-1}^{(N)}
\end{array}\right](z_{\ast}).$$
The probability that the column rank of $D(z_{\ast})$ is $\min(N-1,m-1)$ is equal to $1-O(t)$. This is true because the probability that a full size minor of this matrix is zero at a fixed value $z_{\ast}$ is equal to $O(t)$ if one chooses $d_j^{(i)}$ with $\deg(d_j^{(i)})<n_i$ for $j=1,\hdots m-1$ and $i=1,\hdots,N$ randomly; this follows from Lemma \ref{res} (a) and (b) because that a minor is zero implies that either one of the involved polynomial entries is zero or one of the entries is fixed by the others. Consequently, for $m\geq N$, i.e. $\min(N-1,m-1)=N-1$, the probability that the kernel of $D(z_{\ast})$ is zero is equal to $1-O(t)$. Therefore, the probability that $\xi=\hat{\xi}$ is equal to $1-O(t)$, too.\\
For $m=N-1$, one has $\min(N-1,m-1)=N-2$ and thus, the probability that the kernel has dimension one is equal to $1-O(t)$. This means that only $\xi$ that differ by a nonzero scalar factor lead to the same solution. But if one multiplies $\xi$ by a factor from $\mathbb F^g\setminus\{0\}$, the set of possible values for the $D_i$ which fulfill \eqref{above} does not change, anyway.\\
In summary, with probability $1-O(t)$, one has $\frac{t^{-g(N-1)}-1}{t^{-g}-1}=\sum_{k=0}^{N-2}(t^{-g})^k=t^{-g(N-2)}(1-O(t))$ possibilities for $\xi$ and according to Lemma \ref{res} (a) and (b), for each $\xi$, there are $O(|X(N)|\cdot t^{g(N-1)+m-1})$ possibilities for $\mathcal{D}_N$.
Hence, the probability is $O(t^{g+m-1})=O(t^{m+1})$ for $g\geq 2$. Since one already knows that $f_{z_{\ast}}$ divides $d_m^{(i)}$ for $i=1,\hdots,N$, one has $g\leq \min(n_1,\hdots,n_N)$, i.e. there are only finitely many possibilities for $g$. Consequently, only the case $g=1$ is relevant for the computation of the coefficient of $t^m$. Here, one has $t^{-(N-2)}(1-O(t))$ possibilities for $\xi$ and according to Lemma \ref{res} (a) and (b), $t^{N+m-2}$ possibilities for $z_{\ast}$ and $\mathcal{D}_N$. Hence, the probability of $\tilde{A}_{N+1}(N)$ is $t^m+O(t^{m+1})$.\\
Next, we show that for simple form, it holds $|A_{N+1}(N)|=|\tilde{A}_{N+1}(N)|+O(|X(N)|\cdot t^{m+1})$, which imples $|A_{N+1}(N)|=|X(N)|\cdot O(t^{m+1})$ if $m<N-1$ and $|A_{N+1}(N)|=|X(N)|\cdot(t^m+O(t^{m+1}))$ if $m\geq N-1$, i.e. 
$$|A_{N+1}(N)|=|X(N)|\cdot\left(\sum_{i=N-1}^{\min(m,N-1)}t^m+O(t^{m+1})\right).$$
We prove this by showing that for simple form, $M^C(N):=X(N)\setminus mut(N)$ and $A:=(M^C(N)\setminus A_{N+1}(N))\cap \tilde{A}_{N+1}(N)$, one has $|A|=O(|X(N)|\cdot t^{m+1})$. It holds $(D_1,\hdots, D_N)\in A$ if and only if there exist $z_{\ast}, \tilde{z}_{\ast}\in\overline{\mathbb F}$ such that $d^{(i)}_m(z_{\ast})=0$ for $i=1,\hdots,N$ and the first $m-1$ equations of \eqref{above} are fulfilled for $z_{\ast}$ and there exists a subset of $N-1$ matrices which fulfil equations \eqref{above} at $\tilde{z}_{\ast}$. Since the number of choices for this subset is equal to $N$ and therefore finite, it follows from preceding computations that the probability of $M^C(N)\setminus A_{N+1}(N)$ (i.e. of the condition concerning $\tilde{z}_{\ast}$) is $O(t^m)$. Without restriction, let the mentioned subset be $\{D_1,\hdots,D_{N-1}\}$. If $z_{\ast}=\tilde{z}_{\ast}$, one has, amongst others, the additional condition $d_m^{(N)}(\tilde{z}_{\ast})=0$, which gives a factor  that is $O(t)$ for the probability, according to Lemma \ref{res} (a). If $z_{\ast}\neq \tilde{z}_{\ast}$, one has the additional conditions that $d^{(i)}_m(z_{\ast})=0$ for $i=1,\hdots,N$, which contributes a factor that is $O(t^{N-1})=O(t)$. Consequently, in summary, one has that the probability of $A$ is $O(t^{m+1})$, which is what we wanted to show.

It remains to compute $|A_I(N)|$. From claim 2, 
one already knows $|A_I(N)|=|X(N)|\cdot O(t^m)$. First consider $A_i(N)\cap A_j(N)$, i.e. $I=\{i,j\}$ with $i\neq j$, and assume without restriction $i=1$ and $j=N$. It holds $(D_1,\hdots,D_N)\in A_1(N)\cap A_N(N)$ if and only if $\{D_2,\hdots,D_N\}$ and $\{D_1,\hdots,D_{N-1}\}$ are not mutually left coprime. Since the condition that $\{D_2,\hdots,D_N\}$ are not mutually left coprime causes already a factor for the probability that is $O(t^{m+1})$ if $\mathcal{D}_N$ is not of simple form, it is only necessary to consider simple form. Denote by $\hat{A}_{1,N}(N)$ the subset of $X(N)$ for which $\{D_2,\hdots,D_{N-1}\}$ are not mutually left coprime and write $|A_1(N)\cap A_N(N)|=|A_1(N)\cap A_N(N)\cap \hat{A}_{1,N}(N)|+|A_1(N)\cap A_N(N)\cap \hat{A}_{1,N}^C(N)|$, where $\hat{A}_{1,N}^C(N)$ denotes the complementary set $X(N)\setminus \hat{A}_{1,N}(N)$. Moreover, denote by $\mathcal{D}_N^{(1)}$ the matrix that is achieved if the first $m$ rows and columns of $\mathcal{D}_N$ are deleted. Analogously, denote by $\mathcal{D}_N^{(N)}$ the matrix that is achieved if the last $m$ rows and columns of $\mathcal{D}_N$ are deleted.
If $\mathcal{D}_N\in A_1(N)\cap A_N(N)$ is of simple form, one knows that equations \eqref{above} are valid for $\mathcal{D}_N^{(1)}$ as well as for  $\mathcal{D}_N^{(N)}$. Denote the corresponding $\xi$, $z_{\ast}$ and $g_{z_{\ast}}$ by $\xi^{(1)}, z_{\ast}^{(1)}, g^{(1)}$ and $\xi^{(N)}, z_{\ast}^{(N)}, g^{(N)}$, respectively.
If $\mathcal{D}_N\in A_1(N)\cap A_N(N)\cap \hat{A}_{1,N}^C(N)$, one has $\xi_{N-2}^{(1)}\neq 0$ as well as $\xi_{1}^{(N)}\neq 0$ and therefore, $d_m^{(N)}(z_{\ast}^{(1)})=d_m^{(1)}(z_{\ast}^{(N)})=0$ (see proof of Remark \ref{det}). For fixed $z_{\ast}^{(1)}$ and $z_{\ast}^{(N)}$ (where without restriction $\deg(d_m^{(N)})=n_m\geq g^{(1)}$ and $\deg(d_m^{(1)})=n_1\geq g^{(N)}$ since otherwise, $A_1(N)\cap A_N(N)\cap \hat{A}_{1,N}^C(N)=\emptyset$, anyway), this contributes a factor of $t^{g^{(1)}}$ for the probability that $D_2,\hdots, D_N$ are not mutually left coprime and a factor of $t^{g^{(N)}}$ for the probability that $D_1,\hdots, D_{N-1}$ are not mutually left coprime. Thus, the other equations of $\eqref{above}$ for $\mathcal{D}_N^{(1)}$ contribute a factor that is $O(t^{m-g^{(1)}})$ and the other equations of $\eqref{above}$ for $\mathcal{D}_N^{(N)}$ contribute a factor that is $O(t^{m-g^{(N)}})$. Assume without restriction $g^{(1)}\geq g^{(N)}$. Then, one has a contribution to the probability that is $O(t^{m-g^{(1)}})$ by the equations for $\mathcal{D}_N^{(1)}$ and $\mathcal{D}_N^{(N)}$ but $d_m^{(N)}(z_{\ast}^{(1)})=d_m^{(1)}(z_{\ast}^{(N)})=0$ and the additional factor $t^{g^{(1)}+g^{(N)}}$ for these equations. Hence, in summary, $|A_1(N)\cap A_N(N)\cap \hat{A}_{1,N}^C(N)|=O(|X(N)|\cdot t^{m+g^{(N)}})=O(|X(N)|\cdot t^{m+1})$. Consequently, $|A_1(N)\cap A_N(N)|=|A_1(N)\cap A_N(N)\cap \hat{A}_{1,N}(N)|+O(|X(N)|\cdot t^{m+1})=|\hat{A}_{1,N}(N)|+O(|X(N)|\cdot t^{m+1})$.
Therefore, $|A_i(N)\cap A_j(N)|=|X(N)|\cdot (1-P_m(N-2)+O(t^{m+1}))$ for $i,j\in\{1,\hdots,N\}$ with $i\neq j$.\\
Next, it is shown per induction with respect to $|I|$ that $|A_I(N)|=|\hat{A}_I(N)|+O(|X(N)|\cdot t^{m+1})$ for $2\leq |I|\leq N-2$, where $\hat{A}_I(N)$ denotes the subset of $X(N)$ for which the $N-|I|$ matrices from the set $\{D_i\}_{i\in\{1,\hdots,N\}\setminus I}$ are not mutually left coprime. The proof of the corresponding base clause has already been done in the preceding paragraph. For $|I|=k$ with $3\leq k\leq N-2$, assume without restriction that $I=\{N-k+1,\hdots,N\}$. Since $\hat{A}_I(N)\subset A_I(N)$, one has
\begin{align*}
A_I(N)=&A_{N-k+2,\hdots,N}(N)\cap A_{N-k+1,N-k+3,\hdots,N}(N)=\\
=&\hat{A}_{N-k+2,\hdots,N}(N)\cap \hat{A}_{N-k+1,N-k+3,\hdots,N}(N)+\\
+&(A_{N-k+2,\hdots,N}(N)\cap A_{N-k+1,N-k+3,\hdots,N}(N))\setminus\\
&(\hat{A}_{N-k+2,\hdots,N}(N)\cap \hat{A}_{N-k+1,N-k+3,\hdots,N}(N)).
\end{align*}
Furthermore, 
\begin{align*}
&|(A_{N-k+2,\hdots,N}(N)\cap A_{N-k+1,N-k+3,\hdots,N}(N))\setminus\\
&(\hat{A}_{N-k+2,\hdots,N}(N)\cap \hat{A}_{N-k+1,N-k+3,\hdots,N}(N))|\leq\\
&\leq |(A_{N-k+2,\hdots,N}(N)\cap A_{N-k+1,N-k+3,\hdots,N}(N))\setminus\hat{A}_{N-k+2,\hdots,N}(N)|+\\
&+|(A_{N-k+2,\hdots,N}(N)\cap A_{N-k+1,N-k+3,\hdots,N}(N))\setminus\hat{A}_{N-k+1,N-k+3,\hdots,N}(N)|\leq\\
&\leq |A_{N-k+2,\hdots,N}(N)\setminus\hat{A}_{N-k+2,\hdots,N}(N)|+\\
&+|A_{N-k+1,N-k+3,\hdots,N}(N)\setminus\hat{A}_{N-k+1,N-k+3,\hdots,N}(N)|=O(|X(N)|\cdot t^{m+1})
\end{align*}
per induction since $|\{N-k+2,\hdots,N\}|=|\{N-k+1,N-k+3,\hdots,N\}|=k-1$.
 Moreover, it holds $\hat{A}_{N-k+2,\hdots,N}(N)=\hat{A}_{N-k+2}(N-k+2)=A_{N-k+2}(N-k+2)$ and $\hat{A}_{N-k+1,N-k+3,\hdots,N}(N)=\hat{A}_{N-k+1}(N-k+2)=A_{N-k+1}(N-k+2)$. Consequently, 
 \begin{align}\label{hA}
 |A_I(N)|&=|A_{N-k+2}(N-k+2)\cap A_{N-k+1}(N-k+2)|+O(|X(N)|\cdot t^{m+1})=\nonumber\\
 &=|A_{N-k+1,N-k+2}(N-k+2)|+O(|X(N)|\cdot t^{m+1})=\nonumber\\
 &=|\hat{A}_{N-k+1,N-k+2}(N-k+2)|+O(|X(N)|\cdot t^{m+1})=\nonumber\\
 &=|\hat{A}_I(N)|+O(|X(N)|\cdot t^{m+1}).
 \end{align}
Here, the third equation follows from the base clause.\\
For $|I|=N-1$, assume without loss of generality that $I=\{2,\hdots,N\}$. Analogous to the first line of \eqref{hA} (with setting $k=N-1$), one gets $|A_I(N)|=|S_{12}\cap S_{13}|+O(|X(N)|\cdot t^{m+1})$, where $S_{12}$ and $S_{13}$ are the subsets of $X(N)$ for which $D_1,D_2$ and $D_1,D_3$ are not left coprime, respectively. In the proof of Theorem \ref{pairwise}, it has been shown that $|S_{12}\cap S_{13}|=O(|X(N)|\cdot t^{m+1})$. Therefore, $|A_I|=O(|X(N)|\cdot t^{m+1})$ for $|I|=N-1$ and consequently, $|A_I|=O(|X(N)|\cdot t^{m+1})$ for $|I|=N$, too.\\
In summary, one has $|A_I|=|X(N)|\cdot(1-P_m(N-|I|)+O(t^{m+1}))$ for $|I|\leq N-2$ and $|A_I|=|X(N)|\cdot O(t^{m+1})$ for $|I|\in\{N-1,N\}$.\\
Inserting all achieved results into \eqref{sieb}, using that there are $\binom{N}{|I|}$ subset of $\{1,\hdots,N\}$ with cardinality $|I|$ and dividing by $|X(N)|$ completes the proof of the whole theorem.
\end{proof}
To prove Theorem \ref{mut}, we continue by developing an explicit formula for the coefficient of $t^m$ in $P_m(N)$. Therefore, we write 
$$P_m(N)=1+C(N)\cdot t^m+O(t^{m+1})$$
with coefficients $C(N)\in\mathbb N$, which remain to be computed. From the recursion formula of $P_m(N)$, one can deduce a recursion formula for $C(N)$, which has the following form:
$$C(N)=\sum_{k=1}^{N-2}(-1)^{k+1}\binom{N}{k}C(N-k)-\sum_{i=N-1}^{\min(m,N-1)}1.$$
Solving this recursion formula, one achieves:

\begin{lemma}
$$C(N)=-\sum_{y=2}^{m+1}\binom{N}{y}$$ 
\end{lemma}
\begin{proof}\ \\
We show this formula per induction with respect to $N$.
For $N=2$, one has $-\sum_{y=2}^{m+1}\binom{2}{y}=-1$, which coincides with the result of Theorem \ref{2coprime}.
Moreover, per induction, one knows
\begin{align*}
C(N)&=\sum_{k=1}^{N-2}(-1)^k\binom{N}{k}\sum_{y=2}^{m+1}\binom{N-k}{y}-\sum_{i=N-1}^{\min(m,N-1)}1=\\
&=\sum_{k=1}^{N-2}\sum_{y=2}^{\min(m+1,N-k)}(-1)^k\frac{N!}{k!\cdot y!\cdot (N-k-y)!}-\sum_{i=N-1}^{\min(m,N-1)}1=\\
&=\sum_{y=2}^{\min(m+1,N-1)}\sum_{k=1}^{N-y}(-1)^k\frac{N!}{k!\cdot y!\cdot (N-k-y)!}-\sum_{i=N-1}^{\min(m,N-1)}1
\end{align*}
To simplify this term, one substitutes $M=N-y$ and uses
\begin{align*}
\sum_{k=1}^{M}(-1)^{k}\frac{1}{k!(M-k)!}&=\sum_{k=0}^{M}(-1)^{k}\frac{1}{k!(M-k)!}-\frac{1}{M!}=\\
&=\frac{1}{M!}\left(\sum_{k=0}^{M}(-1)^{k}\binom{M}{k}-1\right)=-\frac{1}{M!},
\end{align*}
which is true since applying the binomial theorem yields $\sum_{k=0}^{M}(-1)^{k}\binom{M}{k}=0$. Hence, one achieves
$$C(N)=-\sum_{y=2}^{\min(m+1,N-1)}\binom{N}{y}-\sum_{i=N-1}^{\min(m,N-1)}1.$$
If $m\leq N-2$, the second sum vanishes and $\min(m+1,N-1)=m+1$. Thus, $C(N)=-\sum_{y=2}^{m+1}\binom{N}{y}$. If $m\geq N-1$, the second sum is equal to $1$ and $\min(m+1,N-1)=N-1$. Hence, one obtains $C(N)=-\sum_{y=2}^{N-1}\binom{N}{y}-1=-\sum_{y=2}^{N}\binom{N}{y}=-\sum_{y=2}^{m+1}\binom{N}{y}$ since $\binom{N}{y}=0$ for $y>N$.
\end{proof}
Finally, we reach the aim of this section and obtain an explicit formula for the probability of mutual left coprimeness:
\begin{theorem}\label{mut}\ \\
For $m, N\geq 2$, the probability that $N$ nonsingular polynomial matrices from $\mathbb F[z]^{m\times m}$ are mutually left coprime is equal to
$$P_m(N)=1-\sum_{y=2}^{m+1}\binom{N}{y}t^m+O(t^{m+1}).$$
\end{theorem}


\section{Application to Parallel Connected Linear Systems and Convolutional Codes}

\subsection{Reachability of Parallel Connected Linear Systems}
\label{sec:2}

The aim of this section is to compute the probability that the parallel connected system 
\begin{align}\label{system}
x_1(\tau+1)=&A_1x_1(\tau)+B_1u(\tau)\nonumber\\
&\vdots\\
x_N(\tau+1)=&A_Nx_N(\tau)+B_Nu(\tau)\nonumber
\end{align}
with state vectors $x_i\in \mathbb{F}^{n_i}$ for $i=1,\ldots,N$ and input
$u\in\mathbb F^m$ is reachable. 
To this end, consider right coprime factorizations $(zI-A_i)^{-1}B_i=P_i(z)Q_i^{-1}(z)$.

\begin{proposition}\label{ab}\cite{Fu-He15}\ \\
The parallel connected system \eqref{system} is reachable if and only if\\
(a) $(A_i,B_i)$ are reachable for $i=1,\hdots,N$ and \\
(b) $Q_1(z),..,Q_N(z)$ are mutually left coprime.
\end{proposition}

Our aim is to count the number of reachable interconnections by counting possible coprime factorizations of the transfer functions of the node systems. Therefore, we need the following statements.

\begin{lemma}\label{Q}\ \\
Let $Q\in\mathbb F[z]^{m\times m}$ nonsingular be in Hermite form with $\deg(\det(Q(z)))=n$. Then, there are exactly $|GL_n(\mathbb F)|$ reachable pairs $(A,B)\in\mathbb F^{n\times n}\times\mathbb F^{n\times m}$ with $(zI-A)^{-1}B=P(z)Q(z)^{-1}$ for some $P\in\mathbb F[z]^{n\times m}$ such that $P$ and $Q$ are right coprime. In other words, there are exactly $|GL_n(\mathbb F)|$ polynomial matrices $P\in\mathbb F[z]^{n\times m}$ such that $P$ and $Q$ are right coprime and $PQ^{-1}$ could be written in the form $P(z)Q(z)^{-1}=(zI-A)^{-1}B$, where $(A,B)\in\mathbb F^{n\times n}\times\mathbb F^{n\times m}$ is reachable.
\end{lemma}
\begin{proof}\ \\
According to Proposition 2.3 of \cite{zab}, there exist a reachable pair $(A,B)$ and a polynomial matrix $P$ that is right coprime to $Q$, such that $(zI-A)^{-1}B=P(z)Q(z)^{-1}$. Now, one considers the orbit of this pair $(A,B)$ under the similarity action on the state space, i.e. the set $\{(TAT^{-1},TB)\ |\ T\in GL_n(\mathbb F)\}$, which clearly consists only of reachable pairs. If $(zI-TAT^{-1})^{-1}TB=\tilde{P}(z)\tilde{Q}(z)^{-1}$ is a right coprime factorization of the transfer function with $\tilde{Q}$ in Hermite form, it follows from Theorem 2.4 a, of \cite{zab} that $Q=\tilde{Q}U$ with a unimodular matrix $U\in GL_n(\mathbb F[z])$. But since the Hermite form of a matrix is unique and $\tilde{Q}$ and $Q$ are both in Hermite form, one knows $\tilde{Q}=Q$. Thus, $Q$ leads to at least $|GL_n(\mathbb F)|$ reachable realizations $(A,B)$. On the other hand, the reverse direction of the statement of Theorem 2.4 a, of \cite{zab} shows that the right coprime factorizations $(zI-A_1)^{-1}B_1=P_1(z)Q(z)^{-1}$ and $(zI-A_2)^{-1}B_2=P_2(z)Q(z)^{-1}$ together with the reachability of $(A_1,B_1)$ and $(A_2,B_2)$ imply $(A_2,B_2)=(TA_1T^{-1},TB_1)$ for some $T\in GL_n(\mathbb F)$. Therefore, $Q$ leads to at most $|GL_n(\mathbb F)|$ reachable realizations $(A,B)$.
\end{proof}

\begin{lemma}\label{nn}\ \\
Let $(A,B)\in\mathbb F^{n\times n}\times\mathbb F^{n\times m}$ and $G(z)=(zI-A)^{-1}B=P(z)Q(z)^{-1}$ be the corresponding transfer function with $P\in\mathbb F[z]^{n\times m}, Q\in\mathbb F[z]^{m\times m}$, where $\det(Q)\not\equiv 0$. Then, the reachability of $(A,B)$ only depends on $P$.
\end{lemma}

\begin{proof}\ \\
By the well-known Kalman test, system $(A,B)$ is reachable if and only if $c=0$ is the only solution of $cA^iB=0$ for $0\leq i\leq n-1$ with $c^{\top}\in\mathbb F^n$. Note that $cA^iB=0$ for $0\leq i\leq n-1$ implies $cA^iB=0$ for $i\geq 0$ by the theorem of Cayley-Hamilton.
Since $(zI-A)^{-1}=\sum_{i=0}^{\infty}\frac{A^i}{z^{i+1}}$,
reachability is equivalent to the fact that $c=0$ is the only solution
of $c(zI-A)^{-1}B\equiv 0$ with $c^{\top}\in\mathbb F^n$. This means that $cP\equiv 0$ for $c^T\in\mathbb F^n$ implies $c=0$, which is a criterion that only depends on $P$.
\end{proof}
Now, we are ready to prove the main theorem of this section:
\begin{theorem}\label{prodws}\ \\
The probability that the parallel connected system given by \eqref{system} is reachable is 
$$\prod_{i=1}^N\prod_{j=m}^{n_i+m-1}(1-t^{j})\cdot P_m(N),$$
where $P_m(N)$ is the probability that $N$ polynomial matrices from $\mathbb F[z]^{m\times m}$ in Hermite form are mutually left coprime.
\end{theorem}

\begin{proof}\ \\
For $i=1,\hdots,N$, consider right coprime factorizations $(zI-A_i)^{-1}B_i=P_i(z)Q_i(z)^{-1}$.
From Theorem \ref{form} (a) one knows that $\deg(\det(Q_i))=n_i$ for $i=1,\hdots,N$ and from Theorem \ref{form} (b) that one could assume that the polynomial matrices $Q_1,\hdots,Q_N$ are in Hermite form. According to Lemma \ref{Q}, for every such $Q_i$, there exist exactly $|GL_{n_i}(\mathbb F)|$ reachable pairs $(A_i,B_i)$. Therefore, the probability that condition (b) of Proposition \ref{ab} is fulfilled is equal to the probability that arbitrary polynomial matrices $Q_i$ (in Hermite form) with $\deg(\det(Q_i))=n_i$ for $i=1,\hdots,N$ are mutually left coprime. Since this condition only depends on the matrices $Q_i$ and according to Lemma \ref{nn}, the reachability of the node systems only depends on $P_i$, one could just multiply the probability of mutual left coprimeness with the probabilities that the node systems are reachable (see Theorem \ref{THMA} for the corresponding formula).
\end{proof}

\begin{remark}\ \\
Since the reachability of the parallel connection of $(A_i,B_i,C_i,D_i)$ is independent of $(C_i,D_i)$ for $i=1,\hdots,N$, the formula of the preceding theorem is also valid if $(C_i,D_i)$ are chosen randomly and are not fixed to $(I,0)$ as in \eqref{system}.
\end{remark}

Finally, we obtain an asymptotic formula for the probability of reachability for a parallel connection.

\begin{theorem}\label{parallel}\ \\
The probability that the parallel connection of $N$ linear systems with $m$ inputs is reachable is equal to
$$1-\sum_{y=1}^{m+1}\binom{N}{y}t^m+O(t^{m+1}).$$
\end{theorem}

\begin{proof}\ \\
Inserting the formula of Theorem \ref{mut} into Theorem \ref{prodws}, leads to 
\begin{align*}
&\prod_{i=1}^N\prod_{j=m}^{n_i+m-1}(1-t^{j})\cdot P_m(N)=\\
&=(1-N\cdot t^m+O(t^{m+1}))\cdot \left(1-\sum_{y=2}^{m+1}\binom{N}{y}t^m+O(t^{m+1})\right)=\\
&=1-\sum_{y=1}^{m+1}\binom{N}{y}t^m+O(t^{m+1}).
\end{align*}
\end{proof}


\subsection{Non-Catastrophic Convolutional Codes}
\label{sec:2}

In this final section, we want to transfer the results of the preceeding sections to convolutional codes. Therefore, we start with a short introduction about convolutional codes and their correlation with linear systems.

\begin{definition}\ \\
A \textbf{convolutional code} $\mathfrak{C}$ of \textbf{rate} $k/n$ is a free $\mathbb F[z]$-submodule of $\mathbb F[z]^n$ of rank $k$.
Hence, there exists $G\in\mathbb F[z]^{n\times k}$ of full column rank such that
$$\mathfrak{C}=\{v\in\mathbb F[z]^n\ |\ v(z)=G(z)m(z)\ \text{for some}\ m\in\mathbb F[z]^k\}.$$
$G$ is called \textbf{generator matrix} of the code and is unique up to right multiplication with a unimodular matrix $U\in Gl_k(\mathbb F[z])$.
\end{definition}

\begin{definition}\ \\
Let $\nu_1, \hdots, \nu_k$ be the column degrees of $G\in\mathbb F[z]^{n\times k}$. Then, $\nu:=\nu_1+\cdots+\nu_k$ is called the \textbf{order} of $G$. The \textbf{degree} $\delta$ of a convolutional code $\mathfrak{C}$ is defined as the minimal order of its generator matrices. Equivalently, one could define the degree of $\mathfrak{C}$ as the maximal degree of the $k\times k$-minors of one and hence, each generator matrix of $\mathfrak{C}$.
\end{definition}


\begin{theorem}\label{nude}\ \\
It holds $\nu=\delta$, i.e. $G$ is a minimal basis of $\mathfrak{C}$, if and only if $G$ is column proper.
\end{theorem}

\begin{definition}\ \\
A convolutional code $\mathfrak{C}$ is called \textbf{non-catastrophic} if one and therefore, each of its generator matrices is right prime.
\end{definition}

In the following, it should be explained, how one could construct a convolutional code based on a linear system (see \cite{RY1999}). To this end, we start with a linear system $(A,B,C,D)\in\mathbb F^{s\times s}\times\mathbb F^{s\times k}\times\mathbb F^{(n-k)\times s}\times\mathbb F^{(n-k)\times k}$ and define $$H(z):=\left[\begin{array}{ccc}
zI-A & 0_{s\times(n-k)} & -B \\ 
-C & I_{n-k} & -D
\end{array}\right].$$
The set of $\begin{pmatrix} y\\u\end{pmatrix}\in\mathbb F[z]^n$ with $y\in\mathbb F[z]^{n-k}$ and $u\in\mathbb F[z]^{k}$ for which there exists $x\in\mathbb F[z]^s$ with
$H(z)\cdot[x(z)\ y(z)\ u(z)]^{\top}=0$
forms a submodule of $\mathbb F[z]^n$ of rank $k$ and thus, a convolutional code of rate $k/n$, which is denoted by $\mathfrak{C}(A,B,C,D)$. Moreover, if one writes $x(z)=x_0z^{\gamma}+\cdots+x_{\gamma}$,  $y(z)=y_0z^{\gamma}+\cdots+y_{\gamma}$ and $u(z)=u_0z^{\gamma}+\cdots+u_{\gamma}$ with $\gamma=\max(\deg(x),\deg(y),\deg(u))$, it holds
\begin{align*}
x_{\tau +1}&=Ax_{\tau}+Bu_{\tau} \\
y_{\tau}&=Cx_{\tau}+Du_{\tau}\\
(x_{\tau}, y_{\tau}, u_{\tau})&=0\ \text{for}\ \tau>\gamma.
\end{align*}
Furthermore, there exist $X\in\mathbb F[z]^{s\times k}, Y\in\mathbb F[z]^{(n-k)\times k}, U\in\mathbb F[z]^{k\times k}$ such that $\operatorname{ker}(H(z))=\operatorname{im}[X(z)^{\top}\ Y(z)^{\top}\ U(z)^{\top}]^{\top}$ and $G(z)=\begin{pmatrix}
Y(z)\\U(z)\end{pmatrix}$ is a generator matrix for $\mathfrak{C}$ with $C(zI-A)^{-1}B+D=Y(z)U(z)^{-1}$.\\
Conversely, for each convolutional code $\mathfrak{C}$ of rate $k/n$ and degree $\delta$, there exists $(A,B,C,D)\in\mathbb F^{s\times s}\times\mathbb F^{s\times k}\times\mathbb F^{(n-k)\times s}\times\mathbb F^{(n-k)\times k}$ with $s\geq\delta$ such that $\mathfrak{C}=\mathfrak{C}(A,B,C,D)$.
Moreover, it is always possible to choose $s=\delta$. In this case, one calls $(A,B,C,D)$ a \textbf{minimal representation} of $\mathfrak{C}$.

\begin{theorem}\cite{RY1999}\ \\
$(A,B,C,D)$ is a minimal representation of $\mathfrak{C}(A,B,C,D)$ if and only if it is reachable.
\end{theorem}

\begin{theorem}\cite{RY1999}\label{cor}\ \\
Assume that $(A,B,C,D)$ is reachable. Then $\mathfrak{C}(A,B,C,D)$ is non-catastrophic if and only if $(A,B,C,D)$ is observable.
\end{theorem}
Since a convolutional code might have different realizations, partly minimal and partly not, we will need the following theorem to be able to compute the probability of non-catastrophicity for a convolutional code.

\begin{theorem}\label{conj}\ \\
If $(A,B,C,D)$ is a minimal representation of a convolutional code $\mathfrak{C}$, the set of all minimal representations of $\mathfrak{C}$ is given by $\{(SAS^{-1},SB,CS^{-1},D)\ |\ S\in Gl_{\delta}(\mathbb F)\}$.
\end{theorem}

\begin{proof}\ \\
Clearly, $(SAS^{-1},SB,CS^{-1},D)$ is a minimal representation of $\mathfrak{C}$. On the other hand, let $(A,B,C,D)$ and $(\tilde{A},\tilde{B},\tilde{C},\tilde{D})$ be minimal representations of $\mathfrak{C}$. Set $K:=\begin{pmatrix}-I\\ 0\end{pmatrix}$, $L:=\begin{pmatrix} A\\ C\end{pmatrix}$ and $M:=\left[\begin{array}{cc} 0 & B\\ -I & D\end{array}\right]$ and define $\tilde{K}$, $\tilde{L}$ and $\tilde{M}$ analogously. It follows from Theorem 3.4 of \cite{ros} that there exist (unique) invertible matrices $S$ and $T$ such that $(\tilde{K},\tilde{L},\tilde{M})=(TKS^{-1},TLS^{-1},TM)$. Write $T=\left[\begin{array}{cc} T_1 & T_2\\ T_3 & T_4\end{array}\right]$. Thus, the first of the preceding equations, implies $T_1=S$ and $T_3=0$. Inserting this into the second equation, leads to $\tilde{A}=SAS^{-1}+T_2CS^{-1}$ and $\tilde{C}=T_4CS^{-1}$. Finally, the third equation yields $T_2=0$, $T_4=I$ and using this $\tilde{B}=SB$ as well as $\tilde{D}=D$.
\end{proof}

With the help of the preceding theorems, it is possible to transfer the probability results for linear systems to probability results for convolutional codes. 

\begin{theorem}\ \\
The probability that a convolutional code of rate $k/n$ and degree $\delta\geq1$ is non-catastrophic is equal to 
\begin{align}\label{ak}
P^{rc}_{n-k,\delta,k}&=\frac{\operatorname{Pr}((A,B,C)\in\mathbb F^{\delta\times\delta}\times\mathbb F^{\delta\times k}\times\mathbb F^{(n-k)\times\delta}\ \text{reachable and observable})}{\operatorname{Pr}((A,B)\in\mathbb F^{\delta\times\delta}\times\mathbb F^{\delta\times k}\ \text{reachable})}\\
&=1-t^{n-k}+O(t^{n-k+1}).\label{ak2}
\end{align}
\end{theorem}


 \begin{proof}\ \\
Equations \eqref{ak} and \eqref{ak2} are simply the statements from Lemma \ref{minl} and Theorem \ref{1}. Hence, it remains to show that the probability of non-catastrophicity is equal to one of the expressions from \eqref{ak}. Consequently, there are two possibilities to prove this theorem.\\
The first way is to show that the probability of non-catastrophicity is equal to $P_{n-k,\delta,k}^{rc}$. From the previous subsection, one knows that there exists\\
$(A,B,C,D)\in\mathbb F^{\delta\times \delta}\times\mathbb F^{\delta\times k}\times\mathbb F^{(n-k)\times \delta}\times\mathbb F^{(n-k)\times k}$ such that $\mathfrak{C}=\mathfrak{C}(A,B,C,D)$ and a generator matrix of $\mathfrak{C}$ of the form $G=\begin{pmatrix}
Y\\U\end{pmatrix}$ with $C(zI-A)^{-1}B+D=Y(z)U(z)^{-1}$. Since $G$ is of full column rank and unimodular equivalent generator matrices define the same convolutional code, one could assume that $U$ is in Kronecker-Hermite form. In particular, it is column proper and because $YU^{-1}$ is proper, it follows from Lemma \ref{degn} that $\deg_j(Y)\leq \deg_j(U)$ for $j=1,\hdots k$. Finally, one knows from Theorem \ref{nude} that $\deg(\det(U))=\delta$. Consequently, $G\in M(n-k,\delta,k)$ (see Definition \ref{M}) and since non-catastrophicity of $\mathfrak{C}$ is equivalent to right primeness of $G$, the statement follows.\\
A second way to prove this theorem is to use Theorem \ref{cor}.
According to Theorem \ref{conj}, each convolutional code of degree $\delta$ has exactly $|GL_{\delta}(\mathbb F)|$ minimal representations $(A,B,C,D)$, i.e. exactly $|GL_{\delta}(\mathbb F)|$ representations with $(A,B)$ reachable; see proof of Lemma \ref{minl}. Moreover, if one of these representations is observable, they are all observable and this is the case if and only if the corresponding code $\mathfrak{C}(A,B,C,D)$ is non-catastrophic. Hence, the probability of non-catastrophicity is equal to the right hand side of equation \eqref{ak}.
\end{proof}

\section{Conclusion}

We calculate the probability that a polynomial matrix of a special structure is right prime as well as the probability that
$N$ polynomial matrices in Hermite form are mutually left coprime. Furthermore, we use these results to obtain asymptotic formulas for the probabilities that a linear system is reachable and observable, that a convolutional code is non-catastrophic as well as for the probability that a parallel connected linear system is reachable.
The correspondence between linear systems and convolutional codes was further investigated in \cite{zerz}, where multidimensional systems and codes over finite rings were considered. It remains an open question for future research to study other correlations between polynomial matrices or linear systems and convolutional codes, e.g. in the field of convolutional network coding \cite{ho}.


\begin{thebibliography}{}
%
%



%
%


\bibitem {fuhr} Fuhrmann PA (1975) On controllability and observability of systems connected in
parallel. IEEE T Circuits Syst 22:57

\bibitem{fu2} Fuhrmann PA (1976) Algebraic system theory: an analyst's point of view. J Franklin Inst 301:521-540

\bibitem{Fu-He15} Fuhrmann PA, Helmke U (2015) The Mathematics of Networks of
  Linear Systems. Springer, New York

\bibitem {gar} Garcia-Armas M, Ghorpade SR, Ram S (2011) Relatively prime polynomials and nonsingular Hankel matrices over finite fields. J Comb Theory A 118.3:819-828
%





\bibitem{hjL}
 Helmke U, Jordan J, Lieb J (2016) Probability estimates for reachability of linear systems defined over finite fields. Adv Math Commun 10(1):63-78

\bibitem {hjl} Helmke U, Jordan J, Lieb J (2014) Reachability of random linear systems over finite fields. In: Pinto R, Malonek PR, Vettori P (eds) Coding Theory and Applications, 4th International Castle Meeting. Palmela Castle, Springer, pp 217-225

\bibitem{hin} Hinrichsen D, Pr\"atzel-Wolters D (1983) Generalized Hermite Matrices and Complete Invariants of Strict System Equivalence. SIAM J. Control Optim. 21(2):289-305

\bibitem{ho}
Ho T,  Lun DS (2008)  Network Coding: An Introduction. Cambridge University Press, New York

%
%
%

\bibitem{lo} Lomadze V, Zerz E (2000) Fractional representations of linear systems. Syst Control Lett 39:275-281

\bibitem{rb} Rosenbrock HH (1970) State-Space and Multivariable Theory. Wiley, New York

\bibitem{ros} Rosenthal J, Schumacher JM, York  EV (199) On Behaviours and Convolutional Codes. IEEE T Inform Theory 42:1881-1891


\bibitem{RY1999} 
    Rosenthal J, York EV (1999) BCH Convolutional Codes. IEEE T Inform Theory 45:1833-1844
 
%
%
%



\bibitem {zab} Zaballa I (1997) Controllability and hermite indices of matrix pairs. Int J Control 68(1):61-68

\bibitem{zerz} Zerz E (2010) On multidimensional convolutional codes and controllability properties of multidimensional systems over finite rings. Asian J Control 12(2):119-126

\end{thebibliography}
\end{document}